\documentclass[10pt]{article}
\usepackage{amsmath, amsthm, amssymb, amsfonts, mathrsfs, fdsymbol, mathtools}
\usepackage{graphicx}
\usepackage{makeidx}
\usepackage[left=2cm,right=2cm,top=2cm,bottom=2cm]{geometry}

\usepackage{bm}
\usepackage{caption}
\usepackage{subcaption}

\usepackage{tikz}
\usetikzlibrary{decorations.pathreplacing}
\tikzstyle{vertex}=[auto=left,circle,draw=black,fill=white, inner sep=1.5]
\usepackage{enumitem}
\usepackage{makecell}
\usepackage{blindtext}

\usepackage{hyperref}
\hypersetup{colorlinks=true, citecolor={blue}, linkcolor=blue, filecolor=magenta, urlcolor=cyan}

\newtheorem{theorem}{Theorem}[section]

\newtheorem{lemm}{Lemma}[section]

\newtheorem{ex}{Example}[section]

\title{Enumeration of Switching Non-isomorphic Signed Wheels}

\author{ Deepak Sehrawat\\
Department of Mathematics\\
Indian Institute of Technology Guwahati\\
Guwahati, India - 781039\\
Email: deepakmath55555@iitg.ac.in\\
\\
Bikash Bhattacharjya\\
Department of Mathematics\\
Indian Institute of Technology Guwahati\\
Guwahati, India - 781039\\
Email: b.bikash@iitg.ac.in
}
\date{}

\begin{document}
\maketitle

\vspace{-0.3in}

\vspace*{0.3in}
\noindent
\textbf{Abstract.} Two signed graphs are called switching isomorphic to each other if one is isomorphic to a switching of the other. The wheel $W_n$ is the join of the cycle $C_n$ and a vertex. For $0 \leq p \leq n$, $\psi_{p}(n)$ is defined to be the number of switching non-isomorphic signed $W_n$ with exactly $p$ negative edges on $C_n$. The number of switching non-isomorphic signed $W_n$ is denoted by $\psi(n)$. In this paper, we compute the values of $\psi_{p}(n)$ for $p=0,1,2,3,4,n-4,n-3,n-2,n-1,n$ and of $\psi(n)$ for $n=4,5,...,10$. Our method of obtaining $\psi_{p}(n)$ not only count the switching non-isomorphic signed wheels but also generates them. 
 
\section{Introduction}\label{introduction}
A \textit{k-ary necklace} of length $n$ is an equivalence class of necklaces, under rotation, formed with $n$ beads which have $k$ available colors. It is known~\cite{Weisstein} that the number $N(n,k)$ of non-equivalent $k$-ary necklaces of length $n$ is given by 

\begin{equation}\label{equ for no of necklaces}
N(n,k) = \frac{1}{n} \sum_{d~ |~ n}\phi(d) k^{\frac{n}{d}} = \frac{1}{n} \sum_{i=1}^n k^{\gcd(n,i)},
\end{equation}

where $\phi$ is Euler's totient function. 

Two necklaces are said to be \textit{isomorphic} if one can be obtained from the other by (cyclic) rotation or reflection. A \textit{k-ary bracelet} of length $n$ is an equivalence class, up to isomorphism, of necklaces of length $n$ with $k$ colors. It is also known~\cite{Weisstein} that the number $N'(n,k)$ of non-equivalent $k$-ary bracelets of length $n$ is given by

\begin{equation}\label{equ for no of braceletes}
N'(n,k) =  \frac{N(n,k) + R(n,k)}{2},
\end{equation}

where 
\begin{equation}
R(n,k) =
\begin{cases}  k^{\frac{n+1}{2}} &\text{if $n$ is odd} \\
(\frac{k+1}{2})k^{\frac{n}{2}} &\text{if $n$ is even}.
\end{cases}
\end{equation}

George P\'{o}lya~\cite{Polya} was the first who discovered a powerful method for enumerating the number of orbits of a group on particular configurations. This method became known as the P\'{o}lya Enumeration Theorem. The numbers $N(n,k)$ and $N'(n,k)$ were determined by finding the number of orbits of the \textit{cyclic group} $Z_n$ and the \textit{dihedral group} $D_{2n}$ on $k$-ary $n$-tuples, respectively. 

Fredricksen and Kessler~\cite{Fredri1} and Fredricksen and Maiorana~\cite{Fredri2} firstly developed an algorithm for generating necklaces. An algorithm for generating $k$-ary bracelets was developed by Joe Sawada~\cite{Sawada}. In best of our knowledge, no other method than algorithm is known for generating bracelets. 

We will see that counting isomorphism type (non-equivalent) of 2-ary bracelets is equivalent to counting of isomorphism types of signed wheels. Our approach for enumerating isomorphism type signed wheels also generates them and does not depend upon any algorithm. 

A \textit{signed graph}, denoted by $\Sigma=(G,\sigma)$, is a graph consisting of an ordinary graph $G$ and a sign function $\sigma : E(G) \rightarrow \{+1,-1\}$ which labels each edge of $G$ as positive or negative. In $\Sigma=(G,\sigma)$, $G$ is called the \textit{underlying graph} of $\Sigma$ and the set $\sigma^{-1}(-1)=\{e\in E(G)~|~\sigma(e)=-1\}$ is called the \textit{signature} of $\Sigma.$ \textit{Switching} $\Sigma$ by a vertex $u$ changes the sign of each edge incident to $u$. Two signed graph $\Sigma_1 =(G,\sigma_1)$ and $\Sigma_2 = (G,\sigma_2)$ are called \textit{switching equivalent} if one can be obtained by a sequence of switchings from the other. If the number of negative edges in a cycle is even then we call the cycle \textit{positive} and \textit{negative}, otherwise.

The following characterization for two signed graphs to be switching equivalent is given by Zaslavsky~\cite{Zaslavsky}.

\begin{lemm}\label{two SGs equivalence}
Two signed graphs $\Sigma_1,\Sigma_2$ are switching equivalent if and only if they have the same set of negative cycles.
\end{lemm}

Given a graph $G$ on $n$ vertices and $m$ edges, there are $2^m$ ways of constructing signed graphs on $G$.

\begin{lemm}\cite{Reza}\label{lemma of switching non-equivalent SGs}
There are $2^{m-n+1}$ switching non-equivalent signed graphs on a connected graph $G$ on $n$ vertices and $m$ edges.  
\end{lemm}

We say the signed graphs $\Sigma_1 =(G,\sigma_1)$ and $\Sigma_2 = (H,\sigma_2)$ are \textit{isomorphic} if there exists a graph isomorphism between $G$ and $H$ preserving the edge signs. Two signed graphs are \textit{switching isomorphic} if one is isomorphic to a switching of the other.

Up to switching isomorphism, it is known that there are two signed $K_3$, three signed $K_4$, and seven signed $K_5$. In \cite{Deepak}, the authors classified all sixteen switching non-isomorphic signed $K_6$. Mallows and Sloane~\cite{Mallows} proved that the number of switching non-isomorphic signed complete graphs on $n$ vertices is equal to the number of Euler graphs on $n$ vertices. In~\cite{Zaslavsky}, Zaslavsky proved that there are only six signed Petersen graphs, up to switching isomorphism.

Recently, Y. Bagheri et al.~\cite{Ramezani} proved that the number of mutually switching non-isomorphic signed graphs associated with a given graph $G$ is equal to the number of orbits of the automorphism group of $G$ acting on the set of all possible signed graphs with underlying graph $G$. In this paper, we have used a different technique to determine the number of switching non-isomorphic signed wheels of some particular orders.

A \textit{wheel}, denoted by $W_n$, is the join of the cycle $C_n$ and a vertex. Let $V(W_n)=\{v,v_1,v_2,...,v_n\}$ and $E(W_n)=\{vv_i, v_iv_{i+1}~|~i=1,2,...,n\}$, where the subscripts are read modulo $n$. For $1 \leq i \leq n$, the edges $vv_i$ are said to be the \textit{spokes} of $W_n$, and the cycle induced by all the edges of form $v_iv_{i+1}$ is said to be the \textit{outer cycle}, denoted by $C_n$, of $W_n$. For $n=3$, the graph $W_3$ is the complete graph $K_4$. It is known that the number of switching non-isomorphic signed graphs over $K_4$ is 3. Thus, in the subsequent discussion, we consider the wheels $W_n$ for $n \geq 4$. 

If a spoke $vv_{j}$, for some $1\leq j \leq n$, is negative in $(W_n, \sigma)$ then one can make it positive by switching $v_j$. Thus for any $(W_n, \sigma)$ there is an equivalent $(W_n, \sigma_1)$ such that $\sigma_1^{-1}(-1) \subseteq E(C_{n})$. The signed wheels whose signatures are subsets of the edges of the outer cycle $C_n$ will be denoted by $(W_n, \sigma)^o$. Also two signatures of $C_n$, with no switching, are isomorphic if and only if the corresponding signed wheels are isomorphic. Therefore, counting of isomorphism types of signed wheels is equivalent to counting isomorphism types of 2-ary bracelets, say bracelets of beads having colors blue and red.

For a fixed $0 \leq p \leq n$, $\psi_{p}(n)$ denotes the number of switching non-isomorphic signed wheels of the form $(W_n, \sigma)^o$ with exactly $p$ negative edges. In other words, $\psi_{p}(n)$ denotes the number of non-equivalent 2-ary bracelets with exactly $p$ red beads and $n-p$ blue beads. By $\psi(n)$, we denote the number of switching non-isomorphic signed wheels on $n+1$ vertices. Thus, $ \psi(n) = \sum\limits_{p=0}^{n} \psi_{p}(n)$.

The values of $\psi_{p}(n)$ for $p=0,1,2,3,4,n-4,n-3,n-2,n-1,n$ are determined in Section~\ref{main results}. Using these values, the values of $\psi(n)$ for $n \leq 10$ are obtained in Section~\ref{exact values}.
 
\section{Terminology and Methodology}\label{methodology} 

Our approach to enumerate the switching non-isomorphic signed wheels is to put $p$ negative edges on $C_n$ at different distances that generate all mutually switching non-isomorphic signed wheels.

By $G_n$, we denote a $regular~n\text{-}gon$ having vertex set $V(G_n)=\{v_{1},v_{2},v_{3},...,v_{n}\}$ and edge set $E(G_n) = \{v_{i}v_{i+1}~|~i=1,2,...,n\}$, where the subscripts are read modulo $n$. 

The \textit{distance} between two vertices $u~\text{and}~v$, denoted $d(u,v)$, in a graph $G$ is defined to be the number of edges in a shortest path between $u~\text{and}~v$. The \textit{distance} between two edges $e_1 = u_1u_2$ and $e_2 = v_1v_2$ in a graph $G$, denoted by $d(e_1, e_2)$, is $\min\{d(u_i, v_j)~ : ~ i \in \{1,2\},~ j \in \{1,2\}\}$. In $G_n$, it is clear that $1 \leq d(v_i,v_j) \leq \lfloor \frac{n}{2} \rfloor$ for all $i \neq j$. Further, if we measure the distance along one particular direction (clockwise or anticlockwise), then we have $1 \leq d(v_i,v_j) \leq n-1$ for all $i \neq j$. 

If $n$ is an even number then the vertices $v_{i}$ and $v_{i+ \frac{n}{2}}$ of $G_n$ are called \textit{diagonally opposite vertices} and the edges $v_{i}v_{i+1}$ and $ v_{i+ \frac{n}{2}}v_{i+1+\frac{n}{2}}$ are called the \textit{opposite edges}. On the other hand, if $n$ is an odd number, the edge $v_{i+\lfloor \frac{n}{2} \rfloor}v_{i+(\lfloor \frac{n}{2} \rfloor+1)}$ is called the \textit{opposite edge} of $v_{i}$ for $1 \leq i \leq n$.

Clearly, $G_n$ features $n$ axes of symmetry. A common point at which all these axes meet is called the \textit{center} of $G_n$. Observe that if $n$ is an even number then half of the axes pass through diagonally opposite vertices and the remaining axes pass through the midpoints of opposite edges. On the other hand, if $n$ is an odd number, all the axes pass through a vertex and the midpoint of its opposite edge. 

Let Aut($G$) denotes the automorphism group of a graph $G$. It is well known that Aut($W_n$) = Aut($G_n$)  = $<\alpha, \beta~|~ \alpha^n = \beta^2 = 1, \alpha\beta = \beta\alpha^{-1}>$, the dihedral group $D_n$.

The following result will be helpful to examine whether two signed wheels are switching equivalent.

\begin{lemm}\label{switching non-equi wheels}
Two signed wheels of the form $(W_n, \sigma)^o$ with different signatures are always switching non-equivalent.
\end{lemm}
\begin{proof}
Let $\Sigma_1=(W_n,\sigma_1)^o$ and $\Sigma_2=(W_n,\sigma_2)^o$ be two signed wheels such that $\sigma_{1}^{-1}(-1) \neq \sigma_{2}^{-1}(-1)$, where $\sigma_{1}^{-1}(-1)$ and $\sigma_{2}^{-1}(-1)$ are subsets of $E(C_n)$. Since each negative edge makes exactly one triangle negative, the result follows from Lemma~\ref{two SGs equivalence}.
\end{proof}
 
Let $\Sigma=(W_n,\sigma)^o$ be a signed wheel with $p$ negative edges. Corresponding to $\Sigma$, we associate an ordered \textit{distance tuple} $D(\Sigma)=(r_{0},r_{1},r_{2},r_{3},...,r_{\lfloor \frac{n}{2} \rfloor})$, where $r_{l}$ denotes the number of distinct pairs of negative edges which are at distance $l$ and $r_{0}+r_{1}+r_{2}+r_{3}+\cdots +r_{\lfloor \frac{n}{2} \rfloor} = \binom{p}{2}$.

\begin{ex}
\rm{Consider $\Sigma = (W_8,\sigma)^o$, as depicted in Figure~\ref{signed wheel W_8}. Let $e_1= v_1v_2, e_2=v_2v_3, e_3=v_4v_5, e_4=v_7v_8$ so that $\sigma^{-1}(-1)= \{e_1,e_2,e_3,e_4\}$. It is easy to see that $d(e_1,e_2)=0,~d(e_1,e_4)=d(e_2,e_3)=1,~\text{and}~d(e_1,e_3)=d(e_2,e_4)=d(e_3,e_4)=2$. Therefore, $r_{0}=1,~r_{1}=2,~r_{2}=3,~r_{3}=0,~r_{4}=0$. Hence we have $D(\Sigma) = (1,2,3,0,0)$.}
\end{ex}

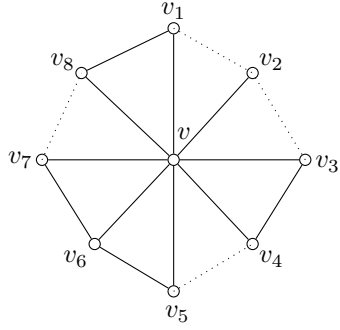
\begin{figure}[ht]
\begin{center}
\begin{tikzpicture}[scale=0.7]
%\draw [help lines] (0,0) grid (85,35);

\node[vertex] (v1) at (10,0) {};
\node[vertex] (v2) at (11.50,0.90) {};
\node[vertex] (v3) at (12.5,2.5) {};
\node[vertex] (v4) at (11.50,4.15) {};
\node[vertex] (v5) at (10,5) {};
\node[vertex] (v6) at (8.25,4.15) {};
\node[vertex] (v7) at (7.5,2.5) {};
\node[vertex] (v8) at (8.50,0.90) {};
\node[vertex] (v9) at (10,2.5) {};

\node [above] at (v5) {$v_{1}$};
\node [xshift=0.25cm, yshift=0.15cm] at (v4) {$v_{2}$};
\node [xshift=0.3cm, yshift=-0.05cm] at (v3) {$v_{3}$};
\node [xshift=0.25cm, yshift=-0.2cm] at (v2) {$v_{4}$};
\node [xshift=0.05cm, yshift=-0.3cm] at (v1) {$v_{5}$};
\node [xshift=-0.25cm, yshift=-0.2cm] at (v8) {$v_{6}$};
\node [left] at (v7) {$v_{7}$};
\node [xshift=-0.25cm, yshift=0.15cm] at (v6) {$v_{8}$};
\node at (10.2,3) {$v$};

\foreach \from/\to in {v2/v3,v5/v6,v7/v8,v8/v1,v9/v1,v9/v2,v9/v3,v9/v4,v9/v5,v9/v6,v9/v7,v9/v8} \draw (\from) -- (\to); 

\draw[dotted] (10,0) -- (11.50,0.90);
\draw[dotted] (10,5) -- (11.50,4.15);
\draw[dotted] (12.5,2.5) -- (11.50,4.15);
\draw[dotted] (7.5,2.5) -- (8.25,4.15);
\end{tikzpicture}
\end{center}
\caption{A signed $W_8$, where dark lines denote positive edges and dotted lines denote negative edges.} \label{signed wheel W_8} 
\end{figure}

The following lemma will help us in deciding whether two signed wheels of the form $(W_n,\sigma)^o$ with $p$ negative edges are isomorphic to each other.  

\begin{lemm}\label{key lemma}
Two signed wheels $\Sigma_1=(W_n,\sigma_1)^o$ and $\Sigma_2= (W_n,\sigma_2)^o$ with $p$ negative edges are isomorphic to each other if and only if $D(\Sigma_1) = D(\Sigma_2)$.
\end{lemm}
\begin{proof}
Let $\Sigma_1$ and $\Sigma_2$ be isomorphic to each other. Since an isomorphism preserve the distance, it follows that $D(\Sigma_1) = D(\Sigma_2)$.

\noindent
Conversely, let $\Sigma_1=(W_n,\sigma_1)^o$ and $\Sigma_2= (W_n,\sigma_2)^o$, with $p$ negative edges, satisfy $D(\Sigma_1) = D(\Sigma_2)$. We need to show that $\Sigma_1$ and $\Sigma_2$ are isomorphic to each other. To establish an isomorphism that maps $\Sigma_1$ onto $\Sigma_2$, we first fix the position of $p$ negative edges of $\Sigma_1$ in clockwise direction, say, at $v_{1_1}v_{1_1 +1}, v_{1_2}v_{1_2 +1},v_{1_3}v_{1_3 +1},...,v_{1_p}v_{1_p +1}$ such that $1 \leq 1_i < 1_j \leq n$ for $1 \leq i < j \leq p$.

Since $D(\Sigma_1) = D(\Sigma_2)$, the positions of $p$ negative edges of $\Sigma_2$ can also be fixed in clockwise direction say, at $v_{2_1}v_{2_1+1}, v_{2_2}v_{2_2+1},v_{2_3}v_{2_3 +1},...,v_{2_p}v_{2_p +1}$, where $1 \leq 2_{i} \leq n$ so that
\begin{equation}\label{equation 1}
d(v_{1_i}v_{1_i +1},v_{1_j}v_{1_j +1})=d(v_{2_i}v_{2_i+1},v_{2_j}v_{2_j +1}),~\text{ for all}~ i,j\in \{1,2,...,p\}.
\end{equation}

Define $\phi : V(W_n) \rightarrow V(W_n)$ by 
$$\phi(x) = \begin{cases} v & \text{if $x=v$}, \\
v_{2_1 + t} & \text{if $x=v_{1_1 + t}$ for $t=0,1,2,...,n-1$}.
\end{cases}$$

It is easy to verify that $\phi$ is an isomorphism that maps $\Sigma_1$ onto $\Sigma_2$. Hence if $D(\Sigma_1) = D(\Sigma_2)$ then $\Sigma_1$ and $\Sigma_2$ are isomorphic to each other.
\end{proof}

Lemma~\ref{switching non-equi wheels} and Lemma~\ref{key lemma} together yield the following result. 

\begin{lemm}\label{main lemma}
Let $\Sigma_1=(W_n,\sigma_1)^o$ and $\Sigma_2= (W_n,\sigma_2)^o$ be two signed wheels with $p$ negative edges such that $D(\Sigma_1) \neq D(\Sigma_2)$. Then $\Sigma_1$ and $\Sigma_2$ are switching non-isomorphic.
\end{lemm}

\begin{lemm}\label{lemma having bounds on dist when p=4}
Among any four edges $e_1, e_2, e_3~\text{and}~e_4$ of $C_n$, there exist two edges $e_i$ and $e_j$ such that $d(e_i,e_j) \leq \lfloor \frac{n-4}{4} \rfloor$.
\end{lemm}
\begin{proof}
For a fix $n$, let if possible 
\begin{equation}
d(e_i,e_j) \geq \Big \lfloor \frac{n-4}{4} \Big \rfloor + 1,~\text{for all}~ i,j \in \{1,2,3,4\},~ i \neq j.
\end{equation}

 If the distance between $e_i~\text{and}~e_j$ is $k$ then there are at least $k-1$ vertices between end vertices of $e_i~\text{and}~e_j$. Therefore there are at least $\lfloor \frac{n-4}{4} \rfloor$ vertices between $e_i~\text{and}~e_j~\text{for all}~ i,j \in \{1,2,3,4\}$. This means there are at least $4\lfloor \frac{n-4}{4} \rfloor +8$ vertices in $C_n$, a contradiction. Hence the result follows.
\end{proof}

Let us place the vertices of $W_n$ in such a way that the outer cycle $C_n$ becomes the regular $n$-gon $G_n$. Let $e_1,e_2,e_3~\text{and}~e_4$ be four negative edges of $(W_n,\sigma)^o$. We place these four edges $e_1,e_2,e_3~\text{and}~e_4$ in such a way that if $i < j$ and $e_i=v_{r}v_{r+1}, e_j=v_{l}v_{l+1}$ then $r+1 \leq l$. Further, in light of Lemma~\ref{lemma having bounds on dist when p=4}, we can always assume that $\text{d}(e_1,e_2) \leq \lfloor \frac{n-4}{4} \rfloor$. Without loss of generality, let $e_1=v_1v_2$. To calculate the value of $\psi_{4}(n)$, we will count different signatures of size four by applying the following strategies. 

\begin{enumerate}
\item[S1.] Take $d(e_1,e_2)=0$ and count the different possibilities for $e_3~\text{and}~e_4$ up to isomorphism. This is carried out in Lemma~\ref{lemma 1 of S1}, Lemma~\ref{lemma 2 of S1}, Lemma~\ref{lemma 3 of S1} and Lemma~\ref{lemma 4 of S1}. 
\item[S2.] Take $d(e_1,e_2)=1$ and count the choices for $e_3~\text{and}~e_4$ under the following conditions:\\
(i) $d(e_2,e_3) \geq 1$;\\
(ii) $d(e_3,e_4) \geq 1$;\\
(iii) $d(e_4,e_1) \geq 1$.\\
Note that if any one of $d(e_2,e_3),d(e_3,e_4),d(e_4,e_1)$ is zero then replacement of those two edges with $e_1~\text{and}~e_2$ will give us a signature which is already encountered in S1. 

\item[S3.] For $d(e_1,e_2)= r$, where $1 \leq r < \lfloor \frac{n-4}{4} \rfloor$, count different choices of $e_3~\text{and}~e_4$.
\item[S4.] If $d(e_1,e_2)= r+1$, where $2 \leq r+1 \leq \lfloor \frac{n-4}{4} \rfloor$, count different choices for $e_3~\text{and}~e_4$ under the following conditions:\\
(i) $d(e_2,e_3) \geq r+1$;\\
(ii) $d(e_3,e_4) \geq r+1$;\\
(iii) $d(e_4,e_1) \geq r+1$.\\
Note that if any one of $d(e_2,e_3),d(e_3,e_4),d(e_4,e_1)$ is less than $r+1$ then replacement of those two edges with $e_1~\text{and}~e_2$ will give us a signature which is already encountered in S3. 
\end{enumerate}

\section{Computation}\label{main results}

In this section, we compute the value of $\psi_{p}(n)$ for $p=0,1,2,3,4,n-4,n-3,n-2,n-1~\text{and}~n$, where $n \geq 4$. To count the number of switching non-isomorphic signed wheels with $p$ negative edges, it is enough to count the different choices of $p$ edges from $E(C_n)$ up to isomorphism (rotations as well as reflections). Note that the counting of different $p$ edges on $C_n$ is same as the counting of different $n-p$ edges. Thus for any $0 \leq p \leq n$, we have 
\begin{equation}\label{main equation}
\psi_{p}(n) = \psi_{n-p}(n).
\end{equation}

The following lemma is trivial.
\begin{lemm}\label{case p=0}
For each $n \geq 4$, $\psi_{0}(n)= \psi_{n}(n) =1$.
\end{lemm}

Any two signed wheels of the form $(W_n, \sigma)^o$ with exactly one negative edge are isomorphic (rotationally equivalent) to each other. Therefore, in the view of Equation~\ref{main equation}, the following lemma is immediate.

\begin{lemm}\label{case p=1}
For each $n \geq 4$, $\psi_{1}(n)=\psi_{n-1}(n)=1$.
\end{lemm}

We now determine the value of $\psi_{n-2}(n)~\text{and}~\psi_{2}(n)$.
\begin{lemm}\label{case p=n-2}
For each $n \geq 4$, $\psi_{2}(n)=\psi_{n-2}(n)=1+\lfloor \frac{n-2}{2} \rfloor$.
\end{lemm}
\begin{proof}
We classify it into two cases.

\begin{enumerate}[wide=0pt]
\item[\textit{Case 1.}] If two edges form a path $P_{3}$ then there is only one possibility up to rotation. One such path is $P_{3}= v_{1}v_{2}v_{3}$.
\item[\textit{Case 2.}] If two edges are disjoint, then the number of choices of two edges among $E(C_n)$ is $\lfloor \frac{n-2}{2} \rfloor$ up to isomorphism.
\end{enumerate}
Each choice of two edges in Case 1 and Case 2 produces a signed wheel $(W_n, \sigma)^o$ with two negative edges. In light of Lemma~\ref{main lemma}, all these signed wheels are mutually switching non-isomorphic. This proves that $\psi_{2}(n)=\psi_{n-2}(n)=1+\lfloor \frac{n-2}{2} \rfloor$.
\end{proof}

A number $n$ is said to have a $k\text{-}partition$ if $n=\lambda_1+\lambda_{2}+...+\lambda_{k}$, where $\lambda_{1} \geq \lambda_{2} \geq \lambda_{3} \geq ...\geq\lambda_{k} \geq 1$. Par($n;k$) denotes the set of all $k$-partitions of $n$ with $p(n;k) = |\text{Par}(n;k)|$. Clearly, the number $p(n;k)$ is zero if $n<k$. The numbers $p(n-3;2)$ and $p(n-3;3)$ are used to compute $\psi_{n-3}(n)$. It is well known that $p(n;3)= \bm{[} \frac{1}{12} n^2 \bm{]}$, where $\bm{[} x \bm{]}$ is the nearest integer function. See \cite{Honsberger} for details.

\begin{lemm}\label{case p=n-3}
For each $n \geq 4$, $\psi_{3}(n)=\psi_{n-3}(n)=1+\lfloor \frac{n-3}{2} \rfloor +\bm{[} \frac{1}{12} (n-3)^2 \bm{]}$.
\end{lemm}
\begin{proof}
Since $n-3$ edges are to be chosen from $C_n$, only following three cases are possible:
\begin{itemize}
\item[(i)] all $n-3$ edges form a path;
\item[(ii)] $n-3$ edges form two different paths;
\item[(iii)] $n-3$ edges form three different paths.
\end{itemize}  
Clearly, there is only one possibility, up to rotation, if $n-3$ edges form a path. For case (ii), the number of two different paths comprising $n-3$ edges is same as the number of partitions of $n-3$ with exactly two parts. Therefore, the number of two such different paths is $\lfloor \frac{n-3}{2} \rfloor$. 

For case (iii), let three distinct paths formed by $n-3$ edges be $P_{t}, P_{t^\prime},~\text{and}~P_{t^{\prime\prime}}$ such that $t \geq t^{\prime}\geq t^{\prime\prime} \geq 2$. For each $t \geq t^{\prime}\geq t^{\prime\prime} \geq 2$, it is easy to see that there is a unique possibility for three such paths, up to rotation. Thus the number of three such paths is same as the number of partitions of $n-3$ with exactly three parts. Hence there are $p(n-3;3)$ different choices for three such paths. 

Each different possibility of $n-3$ edges in cases (i), (ii) and (iii) produces a signed wheel $(W_n, \sigma)^o$ with $n-3$ negative edges and in light of Lemma~\ref{main lemma}, all these signed wheels are mutually switching non-isomorphic. Hence $\psi_{n-3}(n)=\psi_{3}(n)=1+\lfloor \frac{n-3}{2} \rfloor +p(n-3;3)=1+\lfloor \frac{n-3}{2} \rfloor +\bm{[} \frac{1}{12} (n-3)^2 \bm{]}$, as desired. 
\end{proof}

Let $\Sigma=(W_n, \sigma)^o$ be a signed wheel with exactly four negative edges, say $e_1, e_2, e_3,~\text{and}~e_4$. By Lemma~\ref{lemma having bounds on dist when p=4}, it is possible to choose two edges $e_i$ and $e_j$ so that $\text{d}(e_i,e_j) \leq \lfloor \frac{n-4}{4} \rfloor$. Again, a rotation permits us to choose these two edges as $e_1$ and $e_2$ so that $\text{d}(e_1,e_2) \leq \lfloor \frac{n-4}{4} \rfloor$. We now proceed to compute $\psi_{4}(n)$, and to do so we will make use of S1, S2, S3 and S4 .

\begin{lemm}\label{lemma 1 of S1} If edges $e_1, e_2, e_3~\text{and}~e_4$ form a path on $C_n$, then there is only one signed wheel up to rotation. 
\end{lemm}
\begin{lemm}\label{lemma 2 of S1} If edges $e_1, e_2~\text{and}~e_3$ form a path $P_4$ and the edge $e_4$ is at distance at least one from $P_4$, then the number of non-isomorphic signed wheels is $ \lfloor \frac{n}{2} \rfloor -2$.
\end{lemm}

\begin{proof}
Let $e_1=v_1v_2, e_2=v_2v_3~\text{and}~e_3=v_3v_4$. Due to the reflection passing through the mid point of $e_2$, the edge $e_4$ can be $v_{5}v_{6}, v_{6}v_{7},..., v_{\lfloor \frac{n}{2} \rfloor+2}v_{\lfloor \frac{n}{2} \rfloor +3}$ for a total of $\lfloor \frac{n}{2} \rfloor -2$.
\end{proof} 

\begin{lemm}\label{lemma 3 of S1} If the edges $e_1, e_2$ form a path $P_3$ and $e_3,e_4$ form an another path on three vertices disjoint from $P_3$, then the number of non-isomorphic signed wheels is $ \lfloor \frac{n}{2} \rfloor -2$.
\end{lemm}
\begin{proof}
Let $e_1=v_1v_2, e_2=v_2v_3$ and $P_3= v_1v_2v_3$. Let $e_3~\text{and}~e_4$ form an another path $P_3^{\prime}$ disjoint from $P_3$. Due to the reflection passing through $v_2$, the path $P_3^{\prime}$ can be $v_{4}v_{5}v_{6}, v_{5}v_{6}v_{7},...,v_{\lfloor \frac{n}{2} \rfloor +1}v_{\lfloor \frac{n}{2} \rfloor +2}v_{\lfloor \frac{n}{2} \rfloor +3}$ for a total of $\lfloor \frac{n}{2} \rfloor -2$.
\end{proof} 

\begin{lemm}\label{lemma 4 of S1} Let the edges $e_1, e_2$ form a path $P_3$ and $e_3,e_4$ be non-adjacent with each other as well as with $P_3$. Then the number of non-isomorphic signed wheels is $(k-2)^2$ and $(k-3)(k-2)$ when $n=2k+1$ and $n=2k$, respectively.
\end{lemm}
\begin{proof}
Let $e_1=v_1v_2, e_2=v_2v_3$ and $P_3= v_1v_2v_3$. We classify $n$ into two cases.
\begin{enumerate}[wide=0pt]
\item[Case 1.] \textit{Let $n=2k+1$.} If $e_3=v_4v_5$ then due to the reflection passing through $v_{2}$, the edge $e_4$ can be $v_{6}v_{7}, v_{7}v_{8},...,v_{2k}v_{2k+1}$ for a total of $2k-5$.

If $e_3=v_{l}v_{l+1}$ for $5 \leq l \leq k+1$, then the edge $e_4$ can be $v_{l+2}v_{l+3},...,v_{2k-l+4}v_{2k-l+5}$ for a total of $2k-2l+3$. Thus the number of different choices of $e_3$ and $e_4$ is
\vspace*{-0.05in}
\begin{align*}
  & (2k-5)+ \sum_{l=5}^{k+1}[2k-2l+3] \\
=& (2k-5)+2k(k-3)-2 \Big [\frac{(k+1)(k+2)}{2}- 10 \Big ]+3(k-3)\\
=& (k-2)^2.
\end{align*}

\item[Case 2.] \textit{Let $n=2k$.} If $e_3=v_{l}v_{l+1}$ for $4 \leq l \leq k$, then the edge $e_4$ can be $v_{l+2}v_{l+3},...,v_{2k-l+3}v_{2k-l+4}$ for a total of $2k-2l+2$. Thus the number of different choices of $e_3$ and $e_4$ is
\vspace*{-0.05in}
\begin{align*}
  & \sum_{l=4}^{k}[2k-2l+2] \\
=& 2k(k-3)-2 \Big [\frac{k(k+1)}{2}- 6 \Big ]+2(k-3)\\
=& (k-3)(k-2).
\end{align*}
\end{enumerate}

In Case 1 and Case 2, each choice of $e_3$ and $e_4$ along with $P_3$ produces a signed wheel with four negative edges. By Lemma~\ref{key lemma}, all these signed wheels are pairwise non-isomorphic. This completes the proof.
\end{proof}

\begin{lemm}\label{lemma1 p=4}
Let $(W_{2k+1}, \sigma)^o$ be a signed wheel with four negative edges in which $d(e_1,e_2)=r$, where $1 \leq r \leq \lfloor \frac{2k-3}{4} \rfloor$. Then the number of non-isomorphic signed wheels is $[k-(2r+1)]^2$. 
\end{lemm}
\begin{proof}
Let $e_1=v_1v_2$ and $e_2=v_{r+2}v_{r+3}$ such that $d(e_1,e_2)=r$, where $1 \leq r \leq \lfloor \frac{2k-3}{4} \rfloor$. We count the choices for $e_3$ and $e_4$ in the following two cases.

\begin{enumerate}
\item[(i)] If $e_3=v_{2r+3}v_{2r+4}$, then due to the reflection passing through the mid-point of $e_2$, the edge $e_4$ can be $v_{3r+4}v_{3r+5},...,v_{k+r+2}v_{k+r+3}$ for a total of $k-(2r+1).$
\item[(ii)] If $e_3=v_{l}v_{l+1}$, then $e_4$ can be $v_{l+1+r}v_{l+1+r+1},...,v_{(2k+1)-(l-r-3)}v_{(2k+1)-(l-r-3)+1}$ for a total of $2k-2l+4$, where $2r+4 \leq l \leq k+1$.
\end{enumerate} 
Thus if $e_1=v_1v_2$, $e_2=v_{r+2}v_{r+3}$, then the number of choices for $e_3$ and $e_4$ is the sum of all choices obtained in (i) and (ii). Each such choice produces a signed wheel with four negative edges, and by Lemma~\ref{key lemma}, all these signed wheels are mutually non-isomorphic. Hence the number of non-isomorphic signed wheels is
\begin{align*}
  & k-(2r+1)+ \sum_{l=2r+4}^{k+1} (2k-2l+4) \\
=& \{k-(2r+1)\}+ \left[ 2k(k+1-2r-3)-2 \left\{\frac{(k+1)(k+2)}{2}-\frac{(2r+3)(2r+4)}{2} \right\}+4(k+1-2r-3) \right]\\
=&[k-(2r+1)]^2.
\end{align*}
 This completes the proof.
\end{proof}

\begin{lemm}\label{lemma2 p=4}
Let $(W_{2k}, \sigma)^o$ be a signed wheel with four negative edges in which $d(e_1,e_2)=r$, where $1 \leq r \leq \lfloor \frac{2k-4}{4} \rfloor$. Then the number of non-isomorphic signed wheels is $[k-(2r+1)]+[k-(2r+2)]^2$. 
\end{lemm}
\begin{proof}
Let $e_1=v_1v_2$ and $e_2=v_{r+2}v_{r+3}$ such that $d(e_1,e_2)=r$, where $1 \leq r \leq \lfloor \frac{2k-4}{4} \rfloor$. We count the different choices for $e_3$ and $e_4$ in the following two cases:
\begin{enumerate}
\item[(i)] If $e_3=v_{2r+3}v_{2r+4}$, then due to the reflection passing through the mid-point of $e_2$, the edge $e_4$ can be $v_{3r+4}v_{3r+5},...,v_{k+r+2}v_{k+r+3}$ for a total of $k-(2r+1).$
\item[(ii)] If $e_3=v_{l}v_{l+1}$, then $e_4$ can be $v_{l+1+r}v_{l+1+r+1},...,v_{(2k)-(l-r-3)}v_{(2k)-(l-r-3)+1}$ for a total of $2k-2l+3$, where $2r+4 \leq l \leq k+1$.
\end{enumerate} 
Thus the number of non-isomorphic signed wheels is
\begin{align*}
  & k-(2r+1)+ \sum_{l=2r+4}^{k+1} (2k-2l+3) \\
=& \{k-(2r+1)\}+ \left[ 2k(k-2r-2)-2 \left\{\frac{(k+1)(k+2)}{2}-\frac{(2r+3)(2r+4)}{2} \right\}+3(k-2r-2) \right]\\
=& [k-(2r+1)]+[k-(2r+2)]^2.
\end{align*}
This proves the lemma.
\end{proof}

%\textbf{Key Points: 1.} In Lemma~\ref{lemma1 p=4} and Lemma~\ref{lemma2 p=4}, we used the fact that the edge $e_3$ can go from $v_{2r+3}v_{2r+4}$ up to $v_{\lfloor \frac{n}{2} \rfloor +1}v_{\lfloor \frac{n}{2} \rfloor +2}$. This fact works good because if $e_3=v_{l}v_{l+1}$ for some $\lfloor \frac{n}{2} \rfloor+2 \leq l \leq n-2$ and $e_4=v_{r}v_{r+1}$ for some $l+2 \leq r \leq n$ then reflection passing through the mid point of $e_1$ will give us an automorphic signed wheel with four negative edges whose three negative edges will lie from $v_1v_2$ to $v_{\lfloor \frac{n}{2} \rfloor}v_{\lfloor \frac{n}{2} \rfloor +1}$ and such signed wheels are already encountered in some previous step.\\
%\textbf{2.} If $e_3=v_{2r+4}v_{2r+5}$ then $e_4$ goes up to $v_{n-r-1}v_{n-r}$ because if $e_4=v_{j}v_{j+1}$ for some $n-r+1 \leq j \leq n$ then reflection passing through the mid point of $v_{2}~\text{and}~v_{r+2}$ will give us an automorphic signed wheel which has been already encountered. Similarly, for $e_3=v_{l}v_{l+1}$, where $2r+5 \leq l \leq \lfloor \frac{n}{2} \rfloor +1$, edge $e_4$ goes up to $v_{n-(l-r-3)}v_{n-(l-r-3)+1}.$

Note that, in light of Lemma~\ref{main lemma}, all the signed wheels counted in Lemma~\ref{lemma 1 of S1} to Lemma~\ref{lemma2 p=4} are switching non-isomorphic. We now compute $\psi_{4}(n)$ by classifying $n$ into two cases depending upon whether $n$ is odd or even. In the following two theorems, we put $ \lfloor \frac{n-4}{4} \rfloor=l$.

\begin{theorem}\label{Total arrangements of size 4 when n is even}
Let $n = 2k$ for some $k \geq 2$. Then 
\begin{equation}\label{n even and n/4 even}
\psi_{4}(2k) =  (l+1)k^{2} - (2l+3)(l+1)k +  \frac{4l^3+15l^2+20l+9}{3}.
\end{equation}
\end{theorem}
\begin{proof}
Let $\psi_{i}$ be the number of non-isomorphic signed wheels with four negative edges $e_1,e_2,e_3,e_4$ such that $d(e_1,e_2) = i$, where $0 \leq i \leq l$. We have
\begin{align*}
 \psi_{4}(2k) &= \sum_{i=0}^{l} \psi_{i} \\
&= \{1+(k-2)+(k-2)+(k-3)(k-2)\} + \sum_{i=1}^{l} \psi_{i}\\
&= \{k^2-3k+3\} + \sum_{i=1}^{l} [k-(2i+1)]+[k-(2i+2)]^2 \\
&= \{k^2-3k+3\} + \sum_{i=1}^{l} [k^2-3k+6i-4ki+4i^2+3]\\
&= \{k^2-3k+3\} + \Big \{ lk^2-3kl+6\frac{l(l+1)}{2}-4k \frac{l(l+1)}{2}+4 \frac{l(l+1)(2l+1)}{6} +3l \Big \} \\
&= (l+1)k^2-(2l+3)(l+1)k +  \frac{4l^3+15l^2+20l+9}{3}.
\end{align*}
This completes the proof.
\end{proof}

Note that the value of $\psi_{0}$ is the sum of all the values obtained in Lemma~\ref{lemma 1 of S1},~\ref{lemma 2 of S1},~\ref{lemma 3 of S1} and Lemma~\ref{lemma 4 of S1}. For each $1 \leq i \leq l$, the value of $\psi_{i}$ is given in Lemma~\ref{lemma2 p=4}. 

\begin{theorem}\label{Total arrangements of size 4 when n is odd}
Let $n = 2k+1$ for some $k \geq 2$. Then 
\begin{equation}\label{n odd and n/4 even}
\psi_{4}(2k+1) =  (l+1)k^{2} - 2(l+1)^2k +  \frac{(2l+1)(2l+3)(l+1)}{3}.
\end{equation}
\end{theorem}
\begin{proof}
Let $\psi_{i}$ be the number defined in the proof of Theorem~\ref{Total arrangements of size 4 when n is even}. We have
\begin{align*}
 \psi_{4}(2k+1) &= \sum_{i=0}^{l} \psi_{i} \\
&= \{1+(k-2)+(k-2)+(k-2)^2\} + \sum_{i=1}^{l} \psi_{i}\\
&= \{k^2-2k+1\} + \sum_{i=1}^{l} [k-(2i+1)]^2 \\
&= \{k^2-2k+1\} + \sum_{i=1}^{l} [k^2-2k+4i-4ki+4i^2+1]\\
&= \{k^2-2k+1\} + \Big \{ lk^2-2kl+4\frac{l(l+1)}{2}-4k \frac{l(l+1)}{2}+4 \frac{l(l+1)(2l+1)}{6} +l \Big \} \\
&= (l+1)k^{2} - 2(l+1)^2k +  \frac{(2l+1)(2l+3)(l+1)}{3}.
\end{align*}
This completes the proof.
\end{proof}

\section{Main Results}\label{exact values}
In this section, we compute the number of switching non-isomorphic signed wheels $W_{n}$, for $4 \leq n \leq 10$.

\begin{lemm}\label{remaining values}
The value of $\psi_{5}(10)$ is 16.
\end{lemm}
\begin{proof}
To count $\psi_{5}(10)$, the different choices for five edges on $C_{10}$ are considered in the following cases.
\begin{enumerate}
\item If the five edges form a path $P_6$, then there is only one choice for such a path, up to rotation.
\item If the set of five edges is a disjoint union of $P_5~\text{and}~P_2$ then we can assume that $P_5=v_1v_2v_3v_4v_5$. Due to the reflection passing through $v_3$ and $v_8$, the possibilities for $P_2$ are $v_6v_7$ and $v_7v_8$. Therefore there are only two such choices.
\item  If the set of five edges is a disjoint union of $P_4~\text{and}~P_3$, assume that $P_4=v_1v_2v_3v_4$. Due to the reflection passing through the mid point of $v_2v_3$ and its opposite edge $v_7v_8$, the choices for $P_3$ are $v_5v_6v_7$ or $v_6v_7v_8$. Thus there are only two such choices.
\item  If the set of five edges is a disjoint union of $P_4,P_2^{1}~\text{and}~P_2^{2}$, where $P_2^{1}~\text{and}~P_2^{2}$ are paths on two vertices, assume that $P_4=v_1v_2v_3v_4$. Further, if $P_2^{1}=v_5v_6$, then $P_2^{2}$ can be $v_{7}v_{8},v_{8}v_{9},v_{9}v_{10}$. If $P_2^{1}=v_6v_7$, then $P_2^{2}$ must be $v_8v_9$. Hence there are four such choices.
\item If the set of five edges is a disjoint union of $P_3^{1},P_3^{2}~\text{and}~P_2$, where $P_3^{1},P_3^{2}$ are paths on three vertices, assume that $P_3^{1}=v_1v_2v_3$. If $P_3^{2}=v_4v_5v_6$ then due to the reflection passing through the mid point of $v_3v_4$ and its opposite edge $v_8v_9$, $P_2$ can be $v_7v_8~\text{or}~ v_8v_9$. If $P_3^{2}=v_5v_6v_7$ then due to the reflection passing through $v_4$ and its opposite vertex $v_9$, $P_2$ must be $v_8v_9$. Finally, if $P_3^{2}=v_6v_7v_8$ then due to the reflection passing through mid point of $v_4v_5$ and its opposite edge $v_9v_{10}$, $P_2$ must be either $v_4v_5$ or $v_9v_{10}$. Thus there are four choices for this case. 
\item If the set of five edges is a disjoint union of $P_3,P_2^{1},P_2^{2}~\text{and}~P_2^{3}$, where $P_2^{1},P_2^{2}~\text{and}~P_2^{3}$ are paths on two vertices, then there are two such choices, up to automorphism. 
\item If all five edges are mutually disjoint then there is only one choice, up to rotation.
\end{enumerate}
From all these cases, we find that $\psi_{5}(10)=16$. These 16 signed $W_{10}$ are shown in Figure~\ref{psi_5(10)}.
\end{proof}

\begin{lemm}
For $4 \leq n \leq 10$ and $0 \leq p \leq 10$, the values of $\psi_{p}(n)$ are those listed in Table~\ref{table-2}.
\end{lemm}
\begin{proof}
In Table~\ref{table-2}, entries of row $i$, for $i=2,3,4,~\text{and}~5$, are computed from Lemma~\ref{case p=0},~\ref{case p=1},~\ref{case p=n-2}, and Lemma~\ref{case p=n-3} respectively. The values of $\psi_{r}(s)$ for $r=s$ are computed from Lemma~\ref{case p=0} and of $\psi_{r}(s)$ for $r=s-1$ are computed from Lemma~\ref{case p=1}. The values of $\psi_{r-2}(r)$ and $\psi_{r-3}(r)$ for $r=7,8,9,~\text{and}~10$ are computed from Lemma~\ref{case p=n-2} and Lemma~\ref{case p=n-3}, respectively. The values of $\psi_{4}(8)$ and $\psi_{4}(10)(=\psi_{6}(10))$ are computed from Theorem~\ref{Total arrangements of size 4 when n is even} and of $\psi_{4}(9)(=\psi_{5}(9))$ is computed from Theorem~\ref{Total arrangements of size 4 when n is odd}. The value of $\psi_{5}(10)$ is obtained in Lemma~\ref{remaining values}. This proves the lemma.
\end{proof}

\begin{theorem}\label{exact no of signed wheels}
For $n=4,5,6,7,8,9,10$, the number of switching non-isomorphic signed wheels on $W_n$ are those given in Table~\ref{table-3}.
\end{theorem}
\begin{proof}
The values of Table~\ref{table-3} are obtained by respective columns sums of Table~\ref{table-2}.  
\end{proof}

\begin{table}[h]
\begin{center}
\begin{tabular}{|c|c|c|c|c|c|c|c|c|}

\hline
\diaghead(-1,1){aaa}%
{$p$}{$n$} & 4 & 5 & 6 & 7 & 8 & 9 & 10\\
\hline
0 & 1 & 1 & 1 & 1 & 1 & 1 & 1\\
\hline
1 &  1 & 1 & 1 & 1 & 1 & 1 & 1\\
\hline
2 & 2 & 2 & 3 & 3 & 4 & 4 & 5\\
\hline
3 & 1 & 2 & 3 & 4 & 5 & 7 & 8\\
\hline
4 & 1 & 1 & 3 & 4 & 8 & 10 & 16\\
\hline
5 & & 1 & 1 & 3 & 5 & 10 & 16\\
\hline
6 &  &  & 1 & 1 & 4 & 7 & 16\\
\hline
7 &  &  &  & 1 & 1 & 4 & 8\\
\hline
8 &  &  &  &  & 1 & 1 & 5\\
\hline
9 &  &  &  &  &  & 1 & 1\\
\hline
10 &  &  &  &  &  &  & 1\\
\hline
\end{tabular}
\end{center}
\caption{The number $\psi_{p}(n)$ for $n=4,5,...,10$ and $0 \leq p \leq 10$} \label{table-2}
\end{table}

\begin{table}[h]
\begin{center}
\begin{tabular}{|c|c|c|c|c|c|c|c|c|}

\hline
$n$  & 4 & 5 & 6 & 7 & 8 & 9 & 10\\
\hline
$\psi(n)$ & 6 & 8 & 13 & 18 & 30 & 46 & 78\\
\hline

\end{tabular}
\end{center}
\caption{The number of switching non-isomorphic signed $W_n$ for $n=4,5,...,10$} \label{table-3}
\end{table}

\section{Conclusion}
Recall from Lemma~\ref{lemma of switching non-equivalent SGs} that the number of switching non-equivalent signed wheels is $2^{n}$. Another way of getting this number is the following.

It was already noticed that any signed wheel is switching equivalent to a signed wheel whose signature is a subset of $E(C_n)$. Also, by Lemma~\ref{two SGs equivalence}, any two signed wheels whose signatures are different subsets of $E(C_n)$ are switching non-equivalent. As the total number of subsets of $E(C_n)$ are $2^{n}$, there are $2^{n}$ switching non-equivalent signed wheels on $n+1$ vertices. However many of these $2^{n}$ signed wheels are isomorphic to each other. For this purpose, we have determined the value of $\psi_{p}(n)$, for $p=0,1,2,3,4,n-4,n-3,n-2,n-1,n$ and the value of $\psi(n)$, for $n=4,5,6,7,8,9,10$. The values of $\psi_{p}(n)$, for $p=5,6,...,n-5$ are still unknown.

\begin{figure}[h]
\begin{minipage}{0.2\textwidth}
\begin{tikzpicture}[scale=0.22]
\begin{scope}[rotate=18]
        \node[vertex] (v0) at ({360/10*0 }:0.1cm) {};
		\node[vertex] (v3) at ({360/10*0 }:7cm) {};
		\node[vertex] (v2) at ({360/10*1 }:7cm) {};
		\node[vertex] (v1) at ({360/10 *2 }:7cm) {};
		\node[vertex] (v10) at ({360/10 *3 }:7cm) {};
		\node[vertex] (v9) at ({360/10 *4 }:7cm) {};
		\node[vertex] (v8) at ({360/10 *5}:7cm) {};
		\node[vertex] (v7) at ({360/10 *6 }:7cm) {};
		\node[vertex] (v6) at ({360/10 *7}:7cm) {};
		\node[vertex] (v5) at ({360/10 *8}:7cm) {};
		\node[vertex] (v4) at ({360/10 *9}:7cm) {};
        \draw[dotted] ({360/10 *0 +1.5}:7cm) arc  ({360/10*0+1.5}:{360/10*1 -1.5 }:7cm);
        \draw[dotted] ({360/10 *1 +1.5}:7cm) arc  ({360/10*1+1.5}:{360/10*2 -1.5 }:7cm);
        \draw ({360/10 *2 +1.5}:7cm) arc  ({360/10*2+1.5}:{360/10*3 -1.5 }:7cm);
        \draw ({360/10 *3 +1.5}:7cm) arc  ({360/10*3+1.5}:{360/10*4 -1.5 }:7cm);
        \draw ({360/10 *4 +1.5}:7cm) arc  ({360/10*4+1.5}:{360/10*5 -1.5 }:7cm);
        \draw ({360/10 *5 +1.5}:7cm) arc  ({360/10*5+1.5}:{360/10*6 -1.5 }:7cm);
        \draw ({360/10 *6 +1.5}:7cm) arc  ({360/10*6+1.5}:{360/10*7 -1.5 }:7cm);
        \draw[dotted] ({360/10 *7 +1.5}:7cm) arc  ({360/10*7+1.5}:{360/10*8 -1.5 }:7cm);
        \draw[dotted] ({360/10 *8 +1.5}:7cm) arc  ({360/10*8+1.5}:{360/10*9 -1.5 }:7cm);
        \draw[dotted] ({360/10 *9 +1.5}:7cm) arc  ({360/10*9+1.5}:{360/10*10 -1.5 }:7cm);
        
%        \node[xshift=0.2cm, yshift=0.15cm] at ({360/10*0 }:1.5cm) {$v$};
%        \node[xshift=0.25cm] at ({360/10*0 }:7cm) {$v_3$};
%		\node[xshift=0.25cm, yshift=0.1cm] at ({360/10*1 }:7cm) {$v_2$};
%		\node[xshift=0.1cm, yshift=0.25cm] at ({360/10 *2 }:7cm) {$v_1$};
%		\node[xshift=-0.1cm, yshift=0.25cm] at ({360/10 *3 }:7cm) {$v_{10}$};
%		\node[xshift=-0.2cm, yshift=0.15cm] at ({360/10 *4 }:7cm) {$v_9$};
%		\node[xshift=-0.25cm] at ({360/10 *5}:7cm) {$v_8$};
%		\node[xshift=-0.25cm, yshift=-0.15cm] at ({360/10 *6 }:7cm) {$v_7$};
%		\node[xshift=-0.1cm, yshift=-0.3cm] at ({360/10 *7}:7cm) {$v_6$};
%		\node[xshift=0.05cm, yshift=-0.3cm] at ({360/10 *8}:7cm) {$v_5$};
%		\node[xshift=0.2cm, yshift=-0.2cm] at ({360/10 *9}:7cm) {$v_4$};
        \foreach \from/\to in {v0/v1,v0/v2,v0/v3,v0/v4,v0/v5,v0/v6,v0/v7,v0/v8,v0/v9,v0/v10} \draw (\from) -- (\to);
  
  \end{scope}      
        
	\end{tikzpicture}
\end{minipage}
\hfill
\begin{minipage}{0.2\textwidth}
\begin{tikzpicture}[scale=0.22]
   
   \begin{scope}[rotate=18]
     
       \node[vertex] (v0) at ({360/10*0 }:0.1cm) {};
		\node[vertex] (v3) at ({360/10*0 }:7cm) {};
		\node[vertex] (v2) at ({360/10*1 }:7cm) {};
		\node[vertex] (v1) at ({360/10 *2 }:7cm) {};
		\node[vertex] (v10) at ({360/10 *3 }:7cm) {};
		\node[vertex] (v9) at ({360/10 *4 }:7cm) {};
		\node[vertex] (v8) at ({360/10 *5}:7cm) {};
		\node[vertex] (v7) at ({360/10 *6 }:7cm) {};
		\node[vertex] (v6) at ({360/10 *7}:7cm) {};
		\node[vertex] (v5) at ({360/10 *8}:7cm) {};
		\node[vertex] (v4) at ({360/10 *9}:7cm) {};
        \draw[dotted] ({360/10 *0 +1.5}:7cm) arc  ({360/10*0+1.5}:{360/10*1 -1.5 }:7cm);
        \draw[dotted] ({360/10 *1 +1.5}:7cm) arc  ({360/10*1+1.5}:{360/10*2 -1.5 }:7cm);
        \draw ({360/10 *2 +1.5}:7cm) arc  ({360/10*2+1.5}:{360/10*3 -1.5 }:7cm);
        \draw ({360/10 *3 +1.5}:7cm) arc  ({360/10*3+1.5}:{360/10*4 -1.5 }:7cm);
        \draw ({360/10 *4 +1.5}:7cm) arc  ({360/10*4+1.5}:{360/10*5 -1.5 }:7cm);
        \draw ({360/10 *5 +1.5}:7cm) arc  ({360/10*5+1.5}:{360/10*6 -1.5 }:7cm);
        \draw[dotted] ({360/10 *6 +1.5}:7cm) arc  ({360/10*6+1.5}:{360/10*7 -1.5 }:7cm);
        \draw ({360/10 *7 +1.5}:7cm) arc  ({360/10*7+1.5}:{360/10*8 -1.5 }:7cm);
        \draw[dotted] ({360/10 *8 +1.5}:7cm) arc  ({360/10*8+1.5}:{360/10*9 -1.5 }:7cm);
        \draw[dotted] ({360/10 *9 +1.5}:7cm) arc  ({360/10*9+1.5}:{360/10*10 -1.5 }:7cm);
        
%        \node[xshift=0.2cm, yshift=0.15cm] at ({360/10*0 }:1.5cm) {$v$};
%        \node[xshift=0.25cm] at ({360/10*0 }:7cm) {$v_3$};
%		\node[xshift=0.25cm, yshift=0.1cm] at ({360/10*1 }:7cm) {$v_2$};
%		\node[xshift=0.1cm, yshift=0.25cm] at ({360/10 *2 }:7cm) {$v_1$};
%		\node[xshift=-0.1cm, yshift=0.25cm] at ({360/10 *3 }:7cm) {$v_{10}$};
%		\node[xshift=-0.2cm, yshift=0.15cm] at ({360/10 *4 }:7cm) {$v_9$};
%		\node[xshift=-0.25cm] at ({360/10 *5}:7cm) {$v_8$};
%		\node[xshift=-0.25cm, yshift=-0.15cm] at ({360/10 *6 }:7cm) {$v_7$};
%		\node[xshift=-0.1cm, yshift=-0.3cm] at ({360/10 *7}:7cm) {$v_6$};
%		\node[xshift=0.05cm, yshift=-0.3cm] at ({360/10 *8}:7cm) {$v_5$};
%		\node[xshift=0.2cm, yshift=-0.2cm] at ({360/10 *9}:7cm) {$v_4$};
        \foreach \from/\to in {v0/v1,v0/v2,v0/v3,v0/v4,v0/v5,v0/v6,v0/v7,v0/v8,v0/v9,v0/v10} \draw (\from) -- (\to);
        
    \end{scope}    
	\end{tikzpicture}
\end{minipage}
\hfill
\begin{minipage}{0.2\textwidth}
\begin{tikzpicture}[scale=0.22]
\begin{scope}[rotate=18]

		\node[vertex] (v0) at ({360/10*0 }:0.1cm) {};
		\node[vertex] (v3) at ({360/10*0 }:7cm) {};
		\node[vertex] (v2) at ({360/10*1 }:7cm) {};
		\node[vertex] (v1) at ({360/10 *2 }:7cm) {};
		\node[vertex] (v10) at ({360/10 *3 }:7cm) {};
		\node[vertex] (v9) at ({360/10 *4 }:7cm) {};
		\node[vertex] (v8) at ({360/10 *5}:7cm) {};
		\node[vertex] (v7) at ({360/10 *6 }:7cm) {};
		\node[vertex] (v6) at ({360/10 *7}:7cm) {};
		\node[vertex] (v5) at ({360/10 *8}:7cm) {};
		\node[vertex] (v4) at ({360/10 *9}:7cm) {};
        \draw[dotted] ({360/10 *0 +1.5}:7cm) arc  ({360/10*0+1.5}:{360/10*1 -1.5 }:7cm);
        \draw[dotted] ({360/10 *1 +1.5}:7cm) arc  ({360/10*1+1.5}:{360/10*2 -1.5 }:7cm);
        \draw ({360/10 *2 +1.5}:7cm) arc  ({360/10*2+1.5}:{360/10*3 -1.5 }:7cm);
        \draw ({360/10 *3 +1.5}:7cm) arc  ({360/10*3+1.5}:{360/10*4 -1.5 }:7cm);
        \draw ({360/10 *4 +1.5}:7cm) arc  ({360/10*4+1.5}:{360/10*5 -1.5 }:7cm);
        \draw[dotted] ({360/10 *5 +1.5}:7cm) arc  ({360/10*5+1.5}:{360/10*6 -1.5 }:7cm);
        \draw ({360/10 *6 +1.5}:7cm) arc  ({360/10*6+1.5}:{360/10*7 -1.5 }:7cm);
        \draw ({360/10 *7 +1.5}:7cm) arc  ({360/10*7+1.5}:{360/10*8 -1.5 }:7cm);
        \draw[dotted] ({360/10 *8 +1.5}:7cm) arc  ({360/10*8+1.5}:{360/10*9 -1.5 }:7cm);
        \draw[dotted] ({360/10 *9 +1.5}:7cm) arc  ({360/10*9+1.5}:{360/10*10 -1.5 }:7cm);
        
%        \node[xshift=0.2cm, yshift=0.15cm] at ({360/10*0 }:1.5cm) {$v$};
%        \node[xshift=0.25cm] at ({360/10*0 }:7cm) {$v_3$};
%		\node[xshift=0.25cm, yshift=0.1cm] at ({360/10*1 }:7cm) {$v_2$};
%		\node[xshift=0.1cm, yshift=0.25cm] at ({360/10 *2 }:7cm) {$v_1$};
%		\node[xshift=-0.1cm, yshift=0.25cm] at ({360/10 *3 }:7cm) {$v_{10}$};
%		\node[xshift=-0.2cm, yshift=0.15cm] at ({360/10 *4 }:7cm) {$v_9$};
%		\node[xshift=-0.25cm] at ({360/10 *5}:7cm) {$v_8$};
%		\node[xshift=-0.25cm, yshift=-0.15cm] at ({360/10 *6 }:7cm) {$v_7$};
%		\node[xshift=-0.1cm, yshift=-0.3cm] at ({360/10 *7}:7cm) {$v_6$};
%		\node[xshift=0.05cm, yshift=-0.3cm] at ({360/10 *8}:7cm) {$v_5$};
%		\node[xshift=0.2cm, yshift=-0.2cm] at ({360/10 *9}:7cm) {$v_4$};
       \foreach \from/\to in {v0/v1,v0/v2,v0/v3,v0/v4,v0/v5,v0/v6,v0/v7,v0/v8,v0/v9,v0/v10} \draw (\from) -- (\to);
        
    \end{scope}    
	\end{tikzpicture}
\end{minipage}
\hfill
\begin{minipage}{0.2\textwidth}
\begin{tikzpicture}[scale=0.22]

\begin{scope}[rotate=18]

        \node[vertex] (v0) at ({360/10*0 }:0.1cm) {};
		\node[vertex] (v3) at ({360/10*0 }:7cm) {};
		\node[vertex] (v2) at ({360/10*1 }:7cm) {};
		\node[vertex] (v1) at ({360/10 *2 }:7cm) {};
		\node[vertex] (v10) at ({360/10 *3 }:7cm) {};
		\node[vertex] (v9) at ({360/10 *4 }:7cm) {};
		\node[vertex] (v8) at ({360/10 *5}:7cm) {};
		\node[vertex] (v7) at ({360/10 *6 }:7cm) {};
		\node[vertex] (v6) at ({360/10 *7}:7cm) {};
		\node[vertex] (v5) at ({360/10 *8}:7cm) {};
		\node[vertex] (v4) at ({360/10 *9}:7cm) {};
        \draw[dotted] ({360/10 *0 +1.5}:7cm) arc  ({360/10*0+1.5}:{360/10*1 -1.5 }:7cm);
        \draw[dotted] ({360/10 *1 +1.5}:7cm) arc  ({360/10*1+1.5}:{360/10*2 -1.5 }:7cm);
        \draw ({360/10 *2 +1.5}:7cm) arc  ({360/10*2+1.5}:{360/10*3 -1.5 }:7cm);
        \draw ({360/10 *3 +1.5}:7cm) arc  ({360/10*3+1.5}:{360/10*4 -1.5 }:7cm);
        \draw ({360/10 *4 +1.5}:7cm) arc  ({360/10*4+1.5}:{360/10*5 -1.5 }:7cm);
        \draw ({360/10 *5 +1.5}:7cm) arc  ({360/10*5+1.5}:{360/10*6 -1.5 }:7cm);
        \draw[dotted] ({360/10 *6 +1.5}:7cm) arc  ({360/10*6+1.5}:{360/10*7 -1.5 }:7cm);
        \draw[dotted]  ({360/10 *7 +1.5}:7cm) arc  ({360/10*7+1.5}:{360/10*8 -1.5 }:7cm);
        \draw ({360/10 *8 +1.5}:7cm) arc  ({360/10*8+1.5}:{360/10*9 -1.5 }:7cm);
        \draw[dotted] ({360/10 *9 +1.5}:7cm) arc  ({360/10*9+1.5}:{360/10*10 -1.5 }:7cm);
        
%        \node[xshift=0.2cm, yshift=0.15cm] at ({360/10*0 }:1.5cm) {$v$};
%        \node[xshift=0.25cm] at ({360/10*0 }:7cm) {$v_3$};
%		\node[xshift=0.25cm, yshift=0.1cm] at ({360/10*1 }:7cm) {$v_2$};
%		\node[xshift=0.1cm, yshift=0.25cm] at ({360/10 *2 }:7cm) {$v_1$};
%		\node[xshift=-0.1cm, yshift=0.25cm] at ({360/10 *3 }:7cm) {$v_{10}$};
%		\node[xshift=-0.2cm, yshift=0.15cm] at ({360/10 *4 }:7cm) {$v_9$};
%		\node[xshift=-0.25cm] at ({360/10 *5}:7cm) {$v_8$};
%		\node[xshift=-0.25cm, yshift=-0.15cm] at ({360/10 *6 }:7cm) {$v_7$};
%		\node[xshift=-0.1cm, yshift=-0.3cm] at ({360/10 *7}:7cm) {$v_6$};
%		\node[xshift=0.05cm, yshift=-0.3cm] at ({360/10 *8}:7cm) {$v_5$};
%		\node[xshift=0.2cm, yshift=-0.2cm] at ({360/10 *9}:7cm) {$v_4$};
       \foreach \from/\to in {v0/v1,v0/v2,v0/v3,v0/v4,v0/v5,v0/v6,v0/v7,v0/v8,v0/v9,v0/v10} \draw (\from) -- (\to);
        
    \end{scope}    
	\end{tikzpicture}
\end{minipage}
\hfill
\begin{minipage}{0.2\textwidth}
\begin{tikzpicture}[scale=0.22]

\begin{scope}[rotate=18]
		
	\node[vertex] (v0) at ({360/10*0 }:0.1cm) {};	
	 \node[vertex] (v3) at ({360/10*0 }:7cm) {};
		\node[vertex] (v2) at ({360/10*1 }:7cm) {};
		\node[vertex] (v1) at ({360/10 *2 }:7cm) {};
		\node[vertex] (v10) at ({360/10 *3 }:7cm) {};
		\node[vertex] (v9) at ({360/10 *4 }:7cm) {};
		\node[vertex] (v8) at ({360/10 *5}:7cm) {};
		\node[vertex] (v7) at ({360/10 *6 }:7cm) {};
		\node[vertex] (v6) at ({360/10 *7}:7cm) {};
		\node[vertex] (v5) at ({360/10 *8}:7cm) {};
		\node[vertex] (v4) at ({360/10 *9}:7cm) {};
        \draw[dotted] ({360/10 *0 +1.5}:7cm) arc  ({360/10*0+1.5}:{360/10*1 -1.5 }:7cm);
        \draw[dotted] ({360/10 *1 +1.5}:7cm) arc  ({360/10*1+1.5}:{360/10*2 -1.5 }:7cm);
        \draw ({360/10 *2 +1.5}:7cm) arc  ({360/10*2+1.5}:{360/10*3 -1.5 }:7cm);
        \draw ({360/10 *3 +1.5}:7cm) arc  ({360/10*3+1.5}:{360/10*4 -1.5 }:7cm);
        \draw ({360/10 *4 +1.5}:7cm) arc  ({360/10*4+1.5}:{360/10*5 -1.5 }:7cm);
        \draw[dotted] ({360/10 *5 +1.5}:7cm) arc  ({360/10*5+1.5}:{360/10*6 -1.5 }:7cm);
        \draw[dotted] ({360/10 *6 +1.5}:7cm) arc  ({360/10*6+1.5}:{360/10*7 -1.5 }:7cm);
        \draw ({360/10 *7 +1.5}:7cm) arc  ({360/10*7+1.5}:{360/10*8 -1.5 }:7cm);
        \draw ({360/10 *8 +1.5}:7cm) arc  ({360/10*8+1.5}:{360/10*9 -1.5 }:7cm);
        \draw[dotted] ({360/10 *9 +1.5}:7cm) arc  ({360/10*9+1.5}:{360/10*10 -1.5 }:7cm);
        
%        \node[xshift=0.2cm, yshift=0.15cm] at ({360/10*0 }:1.5cm) {$v$};
%        \node[xshift=0.25cm] at ({360/10*0 }:7cm) {$v_3$};
%		\node[xshift=0.25cm, yshift=0.1cm] at ({360/10*1 }:7cm) {$v_2$};
%		\node[xshift=0.1cm, yshift=0.25cm] at ({360/10 *2 }:7cm) {$v_1$};
%		\node[xshift=-0.1cm, yshift=0.25cm] at ({360/10 *3 }:7cm) {$v_{10}$};
%		\node[xshift=-0.2cm, yshift=0.15cm] at ({360/10 *4 }:7cm) {$v_9$};
%		\node[xshift=-0.25cm] at ({360/10 *5}:7cm) {$v_8$};
%		\node[xshift=-0.25cm, yshift=-0.15cm] at ({360/10 *6 }:7cm) {$v_7$};
%		\node[xshift=-0.1cm, yshift=-0.3cm] at ({360/10 *7}:7cm) {$v_6$};
%		\node[xshift=0.05cm, yshift=-0.3cm] at ({360/10 *8}:7cm) {$v_5$};
%		\node[xshift=0.2cm, yshift=-0.2cm] at ({360/10 *9}:7cm) {$v_4$};
       \foreach \from/\to in {v0/v1,v0/v2,v0/v3,v0/v4,v0/v5,v0/v6,v0/v7,v0/v8,v0/v9,v0/v10} \draw (\from) -- (\to);
        
    \end{scope}    
	\end{tikzpicture}
\end{minipage}
\hfill
\begin{minipage}{0.2\textwidth}
\begin{tikzpicture}[scale=0.22]
\begin{scope}[rotate=18]
		\node[vertex] (v0) at ({360/10*0 }:0.1cm) {};
		\node[vertex] (v3) at ({360/10*0 }:7cm) {};
		\node[vertex] (v2) at ({360/10*1 }:7cm) {};
		\node[vertex] (v1) at ({360/10 *2 }:7cm) {};
		\node[vertex] (v10) at ({360/10 *3 }:7cm) {};
		\node[vertex] (v9) at ({360/10 *4 }:7cm) {};
		\node[vertex] (v8) at ({360/10 *5}:7cm) {};
		\node[vertex] (v7) at ({360/10 *6 }:7cm) {};
		\node[vertex] (v6) at ({360/10 *7}:7cm) {};
		\node[vertex] (v5) at ({360/10 *8}:7cm) {};
		\node[vertex] (v4) at ({360/10 *9}:7cm) {};
        \draw[dotted] ({360/10 *0 +1.5}:7cm) arc  ({360/10*0+1.5}:{360/10*1 -1.5 }:7cm);
        \draw[dotted]  ({360/10 *1 +1.5}:7cm) arc  ({360/10*1+1.5}:{360/10*2 -1.5 }:7cm);
        \draw ({360/10 *2 +1.5}:7cm) arc  ({360/10*2+1.5}:{360/10*3 -1.5 }:7cm);
        \draw ({360/10 *3 +1.5}:7cm) arc  ({360/10*3+1.5}:{360/10*4 -1.5 }:7cm);
        \draw ({360/10 *4 +1.5}:7cm) arc  ({360/10*4+1.5}:{360/10*5 -1.5 }:7cm);
        \draw[dotted] ({360/10 *5 +1.5}:7cm) arc  ({360/10*5+1.5}:{360/10*6 -1.5 }:7cm);
        \draw ({360/10 *6 +1.5}:7cm) arc  ({360/10*6+1.5}:{360/10*7 -1.5 }:7cm);
        \draw[dotted] ({360/10 *7 +1.5}:7cm) arc  ({360/10*7+1.5}:{360/10*8 -1.5 }:7cm);
        \draw ({360/10 *8 +1.5}:7cm) arc  ({360/10*8+1.5}:{360/10*9 -1.5 }:7cm);
        \draw[dotted] ({360/10 *9 +1.5}:7cm) arc  ({360/10*9+1.5}:{360/10*10 -1.5 }:7cm);
        
%        \node[xshift=0.2cm, yshift=0.15cm] at ({360/10*0 }:1.5cm) {$v$};
%        \node[xshift=0.25cm] at ({360/10*0 }:7cm) {$v_3$};
%		\node[xshift=0.25cm, yshift=0.1cm] at ({360/10*1 }:7cm) {$v_2$};
%		\node[xshift=0.1cm, yshift=0.25cm] at ({360/10 *2 }:7cm) {$v_1$};
%		\node[xshift=-0.1cm, yshift=0.25cm] at ({360/10 *3 }:7cm) {$v_{10}$};
%		\node[xshift=-0.2cm, yshift=0.15cm] at ({360/10 *4 }:7cm) {$v_9$};
%		\node[xshift=-0.25cm] at ({360/10 *5}:7cm) {$v_8$};
%		\node[xshift=-0.25cm, yshift=-0.15cm] at ({360/10 *6 }:7cm) {$v_7$};
%		\node[xshift=-0.1cm, yshift=-0.3cm] at ({360/10 *7}:7cm) {$v_6$};
%		\node[xshift=0.05cm, yshift=-0.3cm] at ({360/10 *8}:7cm) {$v_5$};
%		\node[xshift=0.2cm, yshift=-0.2cm] at ({360/10 *9}:7cm) {$v_4$};
        \foreach \from/\to in {v0/v1,v0/v2,v0/v3,v0/v4,v0/v5,v0/v6,v0/v7,v0/v8,v0/v9,v0/v10} \draw (\from) -- (\to);
     \end{scope}   
        
	\end{tikzpicture}
\end{minipage}
\hfill
\begin{minipage}{0.2\textwidth}
\begin{tikzpicture}[scale=0.22]
\begin{scope}[rotate=18]
		\node[vertex] (v0) at ({360/10*0 }:0.1cm) {};		
		\node[vertex] (v3) at ({360/10*0 }:7cm) {};
		\node[vertex] (v2) at ({360/10*1 }:7cm) {};
		\node[vertex] (v1) at ({360/10 *2 }:7cm) {};
		\node[vertex] (v10) at ({360/10 *3 }:7cm) {};
		\node[vertex] (v9) at ({360/10 *4 }:7cm) {};
		\node[vertex] (v8) at ({360/10 *5}:7cm) {};
		\node[vertex] (v7) at ({360/10 *6 }:7cm) {};
		\node[vertex] (v6) at ({360/10 *7}:7cm) {};
		\node[vertex] (v5) at ({360/10 *8}:7cm) {};
		\node[vertex] (v4) at ({360/10 *9}:7cm) {};
        \draw[dotted] ({360/10 *0 +1.5}:7cm) arc  ({360/10*0+1.5}:{360/10*1 -1.5 }:7cm);
        \draw[dotted] ({360/10 *1 +1.5}:7cm) arc  ({360/10*1+1.5}:{360/10*2 -1.5 }:7cm);
        \draw ({360/10 *2 +1.5}:7cm) arc  ({360/10*2+1.5}:{360/10*3 -1.5 }:7cm);
        \draw ({360/10 *3 +1.5}:7cm) arc  ({360/10*3+1.5}:{360/10*4 -1.5 }:7cm);
        \draw[dotted] ({360/10 *4 +1.5}:7cm) arc  ({360/10*4+1.5}:{360/10*5 -1.5 }:7cm);
        \draw ({360/10 *5 +1.5}:7cm) arc  ({360/10*5+1.5}:{360/10*6 -1.5 }:7cm);
        \draw ({360/10 *6 +1.5}:7cm) arc  ({360/10*6+1.5}:{360/10*7 -1.5 }:7cm);
        \draw[dotted] ({360/10 *7 +1.5}:7cm) arc  ({360/10*7+1.5}:{360/10*8 -1.5 }:7cm);
        \draw ({360/10 *8 +1.5}:7cm) arc  ({360/10*8+1.5}:{360/10*9 -1.5 }:7cm);
        \draw[dotted] ({360/10 *9 +1.5}:7cm) arc  ({360/10*9+1.5}:{360/10*10 -1.5 }:7cm);
        
%        \node[xshift=0.2cm, yshift=0.15cm] at ({360/10*0 }:1.5cm) {$v$};
%        \node[xshift=0.25cm] at ({360/10*0 }:7cm) {$v_3$};
%		\node[xshift=0.25cm, yshift=0.1cm] at ({360/10*1 }:7cm) {$v_2$};
%		\node[xshift=0.1cm, yshift=0.25cm] at ({360/10 *2 }:7cm) {$v_1$};
%		\node[xshift=-0.1cm, yshift=0.25cm] at ({360/10 *3 }:7cm) {$v_{10}$};
%		\node[xshift=-0.2cm, yshift=0.15cm] at ({360/10 *4 }:7cm) {$v_9$};
%		\node[xshift=-0.25cm] at ({360/10 *5}:7cm) {$v_8$};
%		\node[xshift=-0.25cm, yshift=-0.15cm] at ({360/10 *6 }:7cm) {$v_7$};
%		\node[xshift=-0.1cm, yshift=-0.3cm] at ({360/10 *7}:7cm) {$v_6$};
%		\node[xshift=0.05cm, yshift=-0.3cm] at ({360/10 *8}:7cm) {$v_5$};
%		\node[xshift=0.2cm, yshift=-0.2cm] at ({360/10 *9}:7cm) {$v_4$};
        \foreach \from/\to in {v0/v1,v0/v2,v0/v3,v0/v4,v0/v5,v0/v6,v0/v7,v0/v8,v0/v9,v0/v10} \draw (\from) -- (\to);
        
     \end{scope}   
	\end{tikzpicture}
\end{minipage}
\hfill
\begin{minipage}{0.2\textwidth}
\begin{tikzpicture}[scale=0.22]
\begin{scope}[rotate=18]
		
		\node[vertex] (v0) at ({360/10*0 }:0.1cm) {};
		\node[vertex] (v3) at ({360/10*0 }:7cm) {};
		\node[vertex] (v2) at ({360/10*1 }:7cm) {};
		\node[vertex] (v1) at ({360/10 *2 }:7cm) {};
		\node[vertex] (v10) at ({360/10 *3 }:7cm) {};
		\node[vertex] (v9) at ({360/10 *4 }:7cm) {};
		\node[vertex] (v8) at ({360/10 *5}:7cm) {};
		\node[vertex] (v7) at ({360/10 *6 }:7cm) {};
		\node[vertex] (v6) at ({360/10 *7}:7cm) {};
		\node[vertex] (v5) at ({360/10 *8}:7cm) {};
		\node[vertex] (v4) at ({360/10 *9}:7cm) {};
        \draw[dotted] ({360/10 *0 +1.5}:7cm) arc  ({360/10*0+1.5}:{360/10*1 -1.5 }:7cm);
        \draw[dotted] ({360/10 *1 +1.5}:7cm) arc  ({360/10*1+1.5}:{360/10*2 -1.5 }:7cm);
        \draw ({360/10 *2 +1.5}:7cm) arc  ({360/10*2+1.5}:{360/10*3 -1.5 }:7cm);
        \draw[dotted] ({360/10 *3 +1.5}:7cm) arc  ({360/10*3+1.5}:{360/10*4 -1.5 }:7cm);
        \draw ({360/10 *4 +1.5}:7cm) arc  ({360/10*4+1.5}:{360/10*5 -1.5 }:7cm);
        \draw ({360/10 *5 +1.5}:7cm) arc  ({360/10*5+1.5}:{360/10*6 -1.5 }:7cm);
        \draw ({360/10 *6 +1.5}:7cm) arc  ({360/10*6+1.5}:{360/10*7 -1.5 }:7cm);
        \draw[dotted] ({360/10 *7 +1.5}:7cm) arc  ({360/10*7+1.5}:{360/10*8 -1.5 }:7cm);
        \draw ({360/10 *8 +1.5}:7cm) arc  ({360/10*8+1.5}:{360/10*9 -1.5 }:7cm);
        \draw[dotted] ({360/10 *9 +1.5}:7cm) arc  ({360/10*9+1.5}:{360/10*10 -1.5 }:7cm);
        
%        \node[xshift=0.2cm, yshift=0.15cm] at ({360/10*0 }:1.5cm) {$v$};
%        \node[xshift=0.25cm] at ({360/10*0 }:7cm) {$v_3$};
%		\node[xshift=0.25cm, yshift=0.1cm] at ({360/10*1 }:7cm) {$v_2$};
%		\node[xshift=0.1cm, yshift=0.25cm] at ({360/10 *2 }:7cm) {$v_1$};
%		\node[xshift=-0.1cm, yshift=0.25cm] at ({360/10 *3 }:7cm) {$v_{10}$};
%		\node[xshift=-0.2cm, yshift=0.15cm] at ({360/10 *4 }:7cm) {$v_9$};
%		\node[xshift=-0.25cm] at ({360/10 *5}:7cm) {$v_8$};
%		\node[xshift=-0.25cm, yshift=-0.15cm] at ({360/10 *6 }:7cm) {$v_7$};
%		\node[xshift=-0.1cm, yshift=-0.3cm] at ({360/10 *7}:7cm) {$v_6$};
%		\node[xshift=0.05cm, yshift=-0.3cm] at ({360/10 *8}:7cm) {$v_5$};
%		\node[xshift=0.2cm, yshift=-0.2cm] at ({360/10 *9}:7cm) {$v_4$};
       \foreach \from/\to in {v0/v1,v0/v2,v0/v3,v0/v4,v0/v5,v0/v6,v0/v7,v0/v8,v0/v9,v0/v10} \draw (\from) -- (\to);
        
    \end{scope}    
	\end{tikzpicture}
\end{minipage}
\hfill
\begin{minipage}{0.2\textwidth}
\begin{tikzpicture}[scale=0.22]
\begin{scope}[rotate=18]
		
		\node[vertex] (v0) at ({360/10*0 }:0.1cm) {};
		\node[vertex] (v3) at ({360/10*0 }:7cm) {};
		\node[vertex] (v2) at ({360/10*1 }:7cm) {};
		\node[vertex] (v1) at ({360/10 *2 }:7cm) {};
		\node[vertex] (v10) at ({360/10 *3 }:7cm) {};
		\node[vertex] (v9) at ({360/10 *4 }:7cm) {};
		\node[vertex] (v8) at ({360/10 *5}:7cm) {};
		\node[vertex] (v7) at ({360/10 *6 }:7cm) {};
		\node[vertex] (v6) at ({360/10 *7}:7cm) {};
		\node[vertex] (v5) at ({360/10 *8}:7cm) {};
		\node[vertex] (v4) at ({360/10 *9}:7cm) {};
        \draw[dotted] ({360/10 *0 +1.5}:7cm) arc  ({360/10*0+1.5}:{360/10*1 -1.5 }:7cm);
        \draw[dotted] ({360/10 *1 +1.5}:7cm) arc  ({360/10*1+1.5}:{360/10*2 -1.5 }:7cm);
        \draw ({360/10 *2 +1.5}:7cm) arc  ({360/10*2+1.5}:{360/10*3 -1.5 }:7cm);
        \draw  ({360/10 *3 +1.5}:7cm) arc  ({360/10*3+1.5}:{360/10*4 -1.5 }:7cm);
        \draw[dotted] ({360/10 *4 +1.5}:7cm) arc  ({360/10*4+1.5}:{360/10*5 -1.5 }:7cm);
        \draw  ({360/10 *5 +1.5}:7cm) arc  ({360/10*5+1.5}:{360/10*6 -1.5 }:7cm);
        \draw[dotted] ({360/10 *6 +1.5}:7cm) arc  ({360/10*6+1.5}:{360/10*7 -1.5 }:7cm);
        \draw  ({360/10 *7 +1.5}:7cm) arc  ({360/10*7+1.5}:{360/10*8 -1.5 }:7cm);
        \draw  ({360/10 *8 +1.5}:7cm) arc  ({360/10*8+1.5}:{360/10*9 -1.5 }:7cm);
        \draw[dotted] ({360/10 *9 +1.5}:7cm) arc  ({360/10*9+1.5}:{360/10*10 -1.5 }:7cm);
        
%        \node[xshift=0.2cm, yshift=0.15cm] at ({360/10*0 }:1.5cm) {$v$};
%        \node[xshift=0.25cm] at ({360/10*0 }:7cm) {$v_3$};
%		\node[xshift=0.25cm, yshift=0.1cm] at ({360/10*1 }:7cm) {$v_2$};
%		\node[xshift=0.1cm, yshift=0.25cm] at ({360/10 *2 }:7cm) {$v_1$};
%		\node[xshift=-0.1cm, yshift=0.25cm] at ({360/10 *3 }:7cm) {$v_{10}$};
%		\node[xshift=-0.2cm, yshift=0.15cm] at ({360/10 *4 }:7cm) {$v_9$};
%		\node[xshift=-0.25cm] at ({360/10 *5}:7cm) {$v_8$};
%		\node[xshift=-0.25cm, yshift=-0.15cm] at ({360/10 *6 }:7cm) {$v_7$};
%		\node[xshift=-0.1cm, yshift=-0.3cm] at ({360/10 *7}:7cm) {$v_6$};
%		\node[xshift=0.05cm, yshift=-0.3cm] at ({360/10 *8}:7cm) {$v_5$};
%		\node[xshift=0.2cm, yshift=-0.2cm] at ({360/10 *9}:7cm) {$v_4$};
        \foreach \from/\to in {v0/v1,v0/v2,v0/v3,v0/v4,v0/v5,v0/v6,v0/v7,v0/v8,v0/v9,v0/v10} \draw (\from) -- (\to);
        
     \end{scope}   
	\end{tikzpicture}
\end{minipage}
\hfill
\begin{minipage}{0.2\textwidth}
\begin{tikzpicture}[scale=0.22]
\begin{scope}[rotate=18]
		
		\node[vertex] (v0) at ({360/10*0 }:0.1cm) {};
		\node[vertex] (v3) at ({360/10*0 }:7cm) {};
		\node[vertex] (v2) at ({360/10*1 }:7cm) {};
		\node[vertex] (v1) at ({360/10 *2 }:7cm) {};
		\node[vertex] (v10) at ({360/10 *3 }:7cm) {};
		\node[vertex] (v9) at ({360/10 *4 }:7cm) {};
		\node[vertex] (v8) at ({360/10 *5}:7cm) {};
		\node[vertex] (v7) at ({360/10 *6 }:7cm) {};
		\node[vertex] (v6) at ({360/10 *7}:7cm) {};
		\node[vertex] (v5) at ({360/10 *8}:7cm) {};
		\node[vertex] (v4) at ({360/10 *9}:7cm) {};
        \draw[dotted] ({360/10 *0 +1.5}:7cm) arc  ({360/10*0+1.5}:{360/10*1 -1.5 }:7cm);
        \draw[dotted] ({360/10 *1 +1.5}:7cm) arc  ({360/10*1+1.5}:{360/10*2 -1.5 }:7cm);
        \draw  ({360/10 *2 +1.5}:7cm) arc  ({360/10*2+1.5}:{360/10*3 -1.5 }:7cm);
        \draw  ({360/10 *3 +1.5}:7cm) arc  ({360/10*3+1.5}:{360/10*4 -1.5 }:7cm);
        \draw  ({360/10 *4 +1.5}:7cm) arc  ({360/10*4+1.5}:{360/10*5 -1.5 }:7cm);
        \draw[dotted] ({360/10 *5 +1.5}:7cm) arc  ({360/10*5+1.5}:{360/10*6 -1.5 }:7cm);
        \draw  ({360/10 *6 +1.5}:7cm) arc  ({360/10*6+1.5}:{360/10*7 -1.5 }:7cm);
        \draw[dotted] ({360/10 *7 +1.5}:7cm) arc  ({360/10*7+1.5}:{360/10*8 -1.5 }:7cm);
        \draw[dotted] ({360/10 *8 +1.5}:7cm) arc  ({360/10*8+1.5}:{360/10*9 -1.5 }:7cm);
        \draw ({360/10 *9 +1.5}:7cm) arc  ({360/10*9+1.5}:{360/10*10 -1.5 }:7cm);
        
%        \node[xshift=0.2cm, yshift=0.15cm] at ({360/10*0 }:1.5cm) {$v$};
%        \node[xshift=0.25cm] at ({360/10*0 }:7cm) {$v_3$};
%		\node[xshift=0.25cm, yshift=0.1cm] at ({360/10*1 }:7cm) {$v_2$};
%		\node[xshift=0.1cm, yshift=0.25cm] at ({360/10 *2 }:7cm) {$v_1$};
%		\node[xshift=-0.1cm, yshift=0.25cm] at ({360/10 *3 }:7cm) {$v_{10}$};
%		\node[xshift=-0.2cm, yshift=0.15cm] at ({360/10 *4 }:7cm) {$v_9$};
%		\node[xshift=-0.25cm] at ({360/10 *5}:7cm) {$v_8$};
%		\node[xshift=-0.25cm, yshift=-0.15cm] at ({360/10 *6 }:7cm) {$v_7$};
%		\node[xshift=-0.1cm, yshift=-0.3cm] at ({360/10 *7}:7cm) {$v_6$};
%		\node[xshift=0.05cm, yshift=-0.3cm] at ({360/10 *8}:7cm) {$v_5$};
%		\node[xshift=0.2cm, yshift=-0.2cm] at ({360/10 *9}:7cm) {$v_4$};
        \foreach \from/\to in {v0/v1,v0/v2,v0/v3,v0/v4,v0/v5,v0/v6,v0/v7,v0/v8,v0/v9,v0/v10} \draw (\from) -- (\to);
        
    \end{scope}    
	\end{tikzpicture}
\end{minipage}
\hfill
\begin{minipage}{0.2\textwidth}
\begin{tikzpicture}[scale=0.22]
		\begin{scope}[rotate=18]
		\node[vertex] (v0) at ({360/10*0 }:0.1cm) {};
		\node[vertex] (v3) at ({360/10*0 }:7cm) {};
		\node[vertex] (v2) at ({360/10*1 }:7cm) {};
		\node[vertex] (v1) at ({360/10 *2 }:7cm) {};
		\node[vertex] (v10) at ({360/10 *3 }:7cm) {};
		\node[vertex] (v9) at ({360/10 *4 }:7cm) {};
		\node[vertex] (v8) at ({360/10 *5}:7cm) {};
		\node[vertex] (v7) at ({360/10 *6 }:7cm) {};
		\node[vertex] (v6) at ({360/10 *7}:7cm) {};
		\node[vertex] (v5) at ({360/10 *8}:7cm) {};
		\node[vertex] (v4) at ({360/10 *9}:7cm) {};
        \draw[dotted] ({360/10 *0 +1.5}:7cm) arc  ({360/10*0+1.5}:{360/10*1 -1.5 }:7cm);
        \draw[dotted] ({360/10 *1 +1.5}:7cm) arc  ({360/10*1+1.5}:{360/10*2 -1.5 }:7cm);
        \draw ({360/10 *2 +1.5}:7cm) arc  ({360/10*2+1.5}:{360/10*3 -1.5 }:7cm);
        \draw ({360/10 *3 +1.5}:7cm) arc  ({360/10*3+1.5}:{360/10*4 -1.5 }:7cm);
        \draw[dotted] ({360/10 *4 +1.5}:7cm) arc  ({360/10*4+1.5}:{360/10*5 -1.5 }:7cm);
        \draw ({360/10 *5 +1.5}:7cm) arc  ({360/10*5+1.5}:{360/10*6 -1.5 }:7cm);
        \draw ({360/10 *6 +1.5}:7cm) arc  ({360/10*6+1.5}:{360/10*7 -1.5 }:7cm);
        \draw[dotted] ({360/10 *7 +1.5}:7cm) arc  ({360/10*7+1.5}:{360/10*8 -1.5 }:7cm);
        \draw[dotted] ({360/10 *8 +1.5}:7cm) arc  ({360/10*8+1.5}:{360/10*9 -1.5 }:7cm);
        \draw ({360/10 *9 +1.5}:7cm) arc  ({360/10*9+1.5}:{360/10*10 -1.5 }:7cm);
        
%        \node[xshift=0.2cm, yshift=0.15cm] at ({360/10*0 }:1.5cm) {$v$};
%        \node[xshift=0.25cm] at ({360/10*0 }:7cm) {$v_3$};
%		\node[xshift=0.25cm, yshift=0.1cm] at ({360/10*1 }:7cm) {$v_2$};
%		\node[xshift=0.1cm, yshift=0.25cm] at ({360/10 *2 }:7cm) {$v_1$};
%		\node[xshift=-0.1cm, yshift=0.25cm] at ({360/10 *3 }:7cm) {$v_{10}$};
%		\node[xshift=-0.2cm, yshift=0.15cm] at ({360/10 *4 }:7cm) {$v_9$};
%		\node[xshift=-0.25cm] at ({360/10 *5}:7cm) {$v_8$};
%		\node[xshift=-0.25cm, yshift=-0.15cm] at ({360/10 *6 }:7cm) {$v_7$};
%		\node[xshift=-0.1cm, yshift=-0.3cm] at ({360/10 *7}:7cm) {$v_6$};
%		\node[xshift=0.05cm, yshift=-0.3cm] at ({360/10 *8}:7cm) {$v_5$};
%		\node[xshift=0.2cm, yshift=-0.2cm] at ({360/10 *9}:7cm) {$v_4$};
       \foreach \from/\to in {v0/v1,v0/v2,v0/v3,v0/v4,v0/v5,v0/v6,v0/v7,v0/v8,v0/v9,v0/v10} \draw (\from) -- (\to);
     \end{scope}   
        
	\end{tikzpicture}
\end{minipage}
\hfill
\begin{minipage}{0.2\textwidth}
\begin{tikzpicture}[scale=0.22]
\begin{scope}[rotate=18]
		
		\node[vertex] (v0) at ({360/10*0 }:0.1cm) {};
		\node[vertex] (v3) at ({360/10*0 }:7cm) {};
		\node[vertex] (v2) at ({360/10*1 }:7cm) {};
		\node[vertex] (v1) at ({360/10 *2 }:7cm) {};
		\node[vertex] (v10) at ({360/10 *3 }:7cm) {};
		\node[vertex] (v9) at ({360/10 *4 }:7cm) {};
		\node[vertex] (v8) at ({360/10 *5}:7cm) {};
		\node[vertex] (v7) at ({360/10 *6 }:7cm) {};
		\node[vertex] (v6) at ({360/10 *7}:7cm) {};
		\node[vertex] (v5) at ({360/10 *8}:7cm) {};
		\node[vertex] (v4) at ({360/10 *9}:7cm) {};
        \draw[dotted] ({360/10 *0 +1.5}:7cm) arc  ({360/10*0+1.5}:{360/10*1 -1.5 }:7cm);
        \draw[dotted] ({360/10 *1 +1.5}:7cm) arc  ({360/10*1+1.5}:{360/10*2 -1.5 }:7cm);
        \draw ({360/10 *2 +1.5}:7cm) arc  ({360/10*2+1.5}:{360/10*3 -1.5 }:7cm);
        \draw ({360/10 *3 +1.5}:7cm) arc  ({360/10*3+1.5}:{360/10*4 -1.5 }:7cm);
        \draw[dotted] ({360/10 *4 +1.5}:7cm) arc  ({360/10*4+1.5}:{360/10*5 -1.5 }:7cm);
        \draw ({360/10 *5 +1.5}:7cm) arc  ({360/10*5+1.5}:{360/10*6 -1.5 }:7cm);
        \draw[dotted] ({360/10 *6 +1.5}:7cm) arc  ({360/10*6+1.5}:{360/10*7 -1.5 }:7cm);
        \draw[dotted] ({360/10 *7 +1.5}:7cm) arc  ({360/10*7+1.5}:{360/10*8 -1.5 }:7cm);
        \draw ({360/10 *8 +1.5}:7cm) arc  ({360/10*8+1.5}:{360/10*9 -1.5 }:7cm);
        \draw ({360/10 *9 +1.5}:7cm) arc  ({360/10*9+1.5}:{360/10*10 -1.5 }:7cm);
        
%        \node[xshift=0.2cm, yshift=0.15cm] at ({360/10*0 }:1.5cm) {$v$};
%        \node[xshift=0.25cm] at ({360/10*0 }:7cm) {$v_3$};
%		\node[xshift=0.25cm, yshift=0.1cm] at ({360/10*1 }:7cm) {$v_2$};
%		\node[xshift=0.1cm, yshift=0.25cm] at ({360/10 *2 }:7cm) {$v_1$};
%		\node[xshift=-0.1cm, yshift=0.25cm] at ({360/10 *3 }:7cm) {$v_{10}$};
%		\node[xshift=-0.2cm, yshift=0.15cm] at ({360/10 *4 }:7cm) {$v_9$};
%		\node[xshift=-0.25cm] at ({360/10 *5}:7cm) {$v_8$};
%		\node[xshift=-0.25cm, yshift=-0.15cm] at ({360/10 *6 }:7cm) {$v_7$};
%		\node[xshift=-0.1cm, yshift=-0.3cm] at ({360/10 *7}:7cm) {$v_6$};
%		\node[xshift=0.05cm, yshift=-0.3cm] at ({360/10 *8}:7cm) {$v_5$};
%		\node[xshift=0.2cm, yshift=-0.2cm] at ({360/10 *9}:7cm) {$v_4$};
        \foreach \from/\to in {v0/v1,v0/v2,v0/v3,v0/v4,v0/v5,v0/v6,v0/v7,v0/v8,v0/v9,v0/v10} \draw (\from) -- (\to);
        
    \end{scope}    
	\end{tikzpicture}
\end{minipage}
\hfill
\begin{minipage}{0.2\textwidth}
\begin{tikzpicture}[scale=0.22]
\begin{scope}[rotate=18]
		
		\node[vertex] (v0) at ({360/10*0 }:0.1cm) {};
		\node[vertex] (v3) at ({360/10*0 }:7cm) {};
		\node[vertex] (v2) at ({360/10*1 }:7cm) {};
		\node[vertex] (v1) at ({360/10 *2 }:7cm) {};
		\node[vertex] (v10) at ({360/10 *3 }:7cm) {};
		\node[vertex] (v9) at ({360/10 *4 }:7cm) {};
		\node[vertex] (v8) at ({360/10 *5}:7cm) {};
		\node[vertex] (v7) at ({360/10 *6 }:7cm) {};
		\node[vertex] (v6) at ({360/10 *7}:7cm) {};
		\node[vertex] (v5) at ({360/10 *8}:7cm) {};
		\node[vertex] (v4) at ({360/10 *9}:7cm) {};
        \draw[dotted] ({360/10 *0 +1.5}:7cm) arc  ({360/10*0+1.5}:{360/10*1 -1.5 }:7cm);
        \draw[dotted] ({360/10 *1 +1.5}:7cm) arc  ({360/10*1+1.5}:{360/10*2 -1.5 }:7cm);
        \draw  ({360/10 *2 +1.5}:7cm) arc  ({360/10*2+1.5}:{360/10*3 -1.5 }:7cm);
        \draw  ({360/10 *3 +1.5}:7cm) arc  ({360/10*3+1.5}:{360/10*4 -1.5 }:7cm);
        \draw  ({360/10 *4 +1.5}:7cm) arc  ({360/10*4+1.5}:{360/10*5 -1.5 }:7cm);
        \draw[dotted] ({360/10 *5 +1.5}:7cm) arc  ({360/10*5+1.5}:{360/10*6 -1.5 }:7cm);
        \draw[dotted] ({360/10 *6 +1.5}:7cm) arc  ({360/10*6+1.5}:{360/10*7 -1.5 }:7cm);
        \draw  ({360/10 *7 +1.5}:7cm) arc  ({360/10*7+1.5}:{360/10*8 -1.5 }:7cm);
        \draw[dotted] ({360/10 *8 +1.5}:7cm) arc  ({360/10*8+1.5}:{360/10*9 -1.5 }:7cm);
        \draw  ({360/10 *9 +1.5}:7cm) arc  ({360/10*9+1.5}:{360/10*10 -1.5 }:7cm);
        
%        \node[xshift=0.2cm, yshift=0.15cm] at ({360/10*0 }:1.5cm) {$v$};
%        \node[xshift=0.25cm] at ({360/10*0 }:7cm) {$v_3$};
%		\node[xshift=0.25cm, yshift=0.1cm] at ({360/10*1 }:7cm) {$v_2$};
%		\node[xshift=0.1cm, yshift=0.25cm] at ({360/10 *2 }:7cm) {$v_1$};
%		\node[xshift=-0.1cm, yshift=0.25cm] at ({360/10 *3 }:7cm) {$v_{10}$};
%		\node[xshift=-0.2cm, yshift=0.15cm] at ({360/10 *4 }:7cm) {$v_9$};
%		\node[xshift=-0.25cm] at ({360/10 *5}:7cm) {$v_8$};
%		\node[xshift=-0.25cm, yshift=-0.15cm] at ({360/10 *6 }:7cm) {$v_7$};
%		\node[xshift=-0.1cm, yshift=-0.3cm] at ({360/10 *7}:7cm) {$v_6$};
%		\node[xshift=0.05cm, yshift=-0.3cm] at ({360/10 *8}:7cm) {$v_5$};
%		\node[xshift=0.2cm, yshift=-0.2cm] at ({360/10 *9}:7cm) {$v_4$};
        \foreach \from/\to in {v0/v1,v0/v2,v0/v3,v0/v4,v0/v5,v0/v6,v0/v7,v0/v8,v0/v9,v0/v10} \draw (\from) -- (\to);
        
    \end{scope}    
	\end{tikzpicture}
\end{minipage}
\hfill
\begin{minipage}{0.2\textwidth}
\begin{tikzpicture}[scale=0.22]
\begin{scope}[rotate=18]
		
		\node[vertex] (v0) at ({360/10*0 }:0.1cm) {};
		\node[vertex] (v3) at ({360/10*0 }:7cm) {};
		\node[vertex] (v2) at ({360/10*1 }:7cm) {};
		\node[vertex] (v1) at ({360/10 *2 }:7cm) {};
		\node[vertex] (v10) at ({360/10 *3 }:7cm) {};
		\node[vertex] (v9) at ({360/10 *4 }:7cm) {};
		\node[vertex] (v8) at ({360/10 *5}:7cm) {};
		\node[vertex] (v7) at ({360/10 *6 }:7cm) {};
		\node[vertex] (v6) at ({360/10 *7}:7cm) {};
		\node[vertex] (v5) at ({360/10 *8}:7cm) {};
		\node[vertex] (v4) at ({360/10 *9}:7cm) {};
        \draw[dotted] ({360/10 *0 +1.5}:7cm) arc  ({360/10*0+1.5}:{360/10*1 -1.5 }:7cm);
        \draw[dotted] ({360/10 *1 +1.5}:7cm) arc  ({360/10*1+1.5}:{360/10*2 -1.5 }:7cm);
        \draw ({360/10 *2 +1.5}:7cm) arc  ({360/10*2+1.5}:{360/10*3 -1.5 }:7cm);
        \draw ({360/10 *3 +1.5}:7cm) arc  ({360/10*3+1.5}:{360/10*4 -1.5 }:7cm);
        \draw[dotted] ({360/10 *4 +1.5}:7cm) arc  ({360/10*4+1.5}:{360/10*5 -1.5 }:7cm);
        \draw ({360/10 *5 +1.5}:7cm) arc  ({360/10*5+1.5}:{360/10*6 -1.5 }:7cm);
        \draw[dotted] ({360/10 *6 +1.5}:7cm) arc  ({360/10*6+1.5}:{360/10*7 -1.5 }:7cm);
        \draw ({360/10 *7 +1.5}:7cm) arc  ({360/10*7+1.5}:{360/10*8 -1.5 }:7cm);
        \draw[dotted] ({360/10 *8 +1.5}:7cm) arc  ({360/10*8+1.5}:{360/10*9 -1.5 }:7cm);
        \draw ({360/10 *9 +1.5}:7cm) arc  ({360/10*9+1.5}:{360/10*10 -1.5 }:7cm);
        
%        \node[xshift=0.2cm, yshift=0.15cm] at ({360/10*0 }:1.5cm) {$v$};
%        \node[xshift=0.25cm] at ({360/10*0 }:7cm) {$v_3$};
%		\node[xshift=0.25cm, yshift=0.1cm] at ({360/10*1 }:7cm) {$v_2$};
%		\node[xshift=0.1cm, yshift=0.25cm] at ({360/10 *2 }:7cm) {$v_1$};
%		\node[xshift=-0.1cm, yshift=0.25cm] at ({360/10 *3 }:7cm) {$v_{10}$};
%		\node[xshift=-0.2cm, yshift=0.15cm] at ({360/10 *4 }:7cm) {$v_9$};
%		\node[xshift=-0.25cm] at ({360/10 *5}:7cm) {$v_8$};
%		\node[xshift=-0.25cm, yshift=-0.15cm] at ({360/10 *6 }:7cm) {$v_7$};
%		\node[xshift=-0.1cm, yshift=-0.3cm] at ({360/10 *7}:7cm) {$v_6$};
%		\node[xshift=0.05cm, yshift=-0.3cm] at ({360/10 *8}:7cm) {$v_5$};
%		\node[xshift=0.2cm, yshift=-0.2cm] at ({360/10 *9}:7cm) {$v_4$};
        \foreach \from/\to in {v0/v1,v0/v2,v0/v3,v0/v4,v0/v5,v0/v6,v0/v7,v0/v8,v0/v9,v0/v10} \draw (\from) -- (\to);
        
    \end{scope}    
	\end{tikzpicture}
\end{minipage}
\hfill
\begin{minipage}{0.2\textwidth}
\begin{tikzpicture}[scale=0.22]
\begin{scope}[rotate=18]
		
		\node[vertex] (v0) at ({360/10*0 }:0.1cm) {};
		\node[vertex] (v3) at ({360/10*0 }:7cm) {};
		\node[vertex] (v2) at ({360/10*1 }:7cm) {};
		\node[vertex] (v1) at ({360/10 *2 }:7cm) {};
		\node[vertex] (v10) at ({360/10 *3 }:7cm) {};
		\node[vertex] (v9) at ({360/10 *4 }:7cm) {};
		\node[vertex] (v8) at ({360/10 *5}:7cm) {};
		\node[vertex] (v7) at ({360/10 *6 }:7cm) {};
		\node[vertex] (v6) at ({360/10 *7}:7cm) {};
		\node[vertex] (v5) at ({360/10 *8}:7cm) {};
		\node[vertex] (v4) at ({360/10 *9}:7cm) {};
        \draw[dotted] ({360/10 *0 +1.5}:7cm) arc  ({360/10*0+1.5}:{360/10*1 -1.5 }:7cm);
        \draw[dotted] ({360/10 *1 +1.5}:7cm) arc  ({360/10*1+1.5}:{360/10*2 -1.5 }:7cm);
        \draw ({360/10 *2 +1.5}:7cm) arc  ({360/10*2+1.5}:{360/10*3 -1.5 }:7cm);
        \draw[dotted] ({360/10 *3 +1.5}:7cm) arc  ({360/10*3+1.5}:{360/10*4 -1.5 }:7cm);
        \draw ({360/10 *4 +1.5}:7cm) arc  ({360/10*4+1.5}:{360/10*5 -1.5 }:7cm);
        \draw ({360/10 *5 +1.5}:7cm) arc  ({360/10*5+1.5}:{360/10*6 -1.5 }:7cm);
        \draw[dotted] ({360/10 *6 +1.5}:7cm) arc  ({360/10*6+1.5}:{360/10*7 -1.5 }:7cm);
        \draw ({360/10 *7 +1.5}:7cm) arc  ({360/10*7+1.5}:{360/10*8 -1.5 }:7cm);
        \draw[dotted] ({360/10 *8 +1.5}:7cm) arc  ({360/10*8+1.5}:{360/10*9 -1.5 }:7cm);
        \draw ({360/10 *9 +1.5}:7cm) arc  ({360/10*9+1.5}:{360/10*10 -1.5 }:7cm);
        
%        \node[xshift=0.2cm, yshift=0.15cm] at ({360/10*0 }:1.5cm) {$v$};
%        \node[xshift=0.25cm] at ({360/10*0 }:7cm) {$v_3$};
%		\node[xshift=0.25cm, yshift=0.1cm] at ({360/10*1 }:7cm) {$v_2$};
%		\node[xshift=0.1cm, yshift=0.25cm] at ({360/10 *2 }:7cm) {$v_1$};
%		\node[xshift=-0.1cm, yshift=0.25cm] at ({360/10 *3 }:7cm) {$v_{10}$};
%		\node[xshift=-0.2cm, yshift=0.15cm] at ({360/10 *4 }:7cm) {$v_9$};
%		\node[xshift=-0.25cm] at ({360/10 *5}:7cm) {$v_8$};
%		\node[xshift=-0.25cm, yshift=-0.15cm] at ({360/10 *6 }:7cm) {$v_7$};
%		\node[xshift=-0.1cm, yshift=-0.3cm] at ({360/10 *7}:7cm) {$v_6$};
%		\node[xshift=0.05cm, yshift=-0.3cm] at ({360/10 *8}:7cm) {$v_5$};
%		\node[xshift=0.2cm, yshift=-0.2cm] at ({360/10 *9}:7cm) {$v_4$};
        \foreach \from/\to in {v0/v1,v0/v2,v0/v3,v0/v4,v0/v5,v0/v6,v0/v7,v0/v8,v0/v9,v0/v10} \draw (\from) -- (\to);
        
     \end{scope}   
	\end{tikzpicture}
\end{minipage}
\hfill
\begin{minipage}{0.2\textwidth}
\begin{tikzpicture}[scale=0.22]
\begin{scope}[rotate=18]
		
		\node[vertex] (v0) at ({360/10*0 }:0.1cm) {};
		\node[vertex] (v3) at ({360/10*0 }:7cm) {};
		\node[vertex] (v2) at ({360/10*1 }:7cm) {};
		\node[vertex] (v1) at ({360/10 *2 }:7cm) {};
		\node[vertex] (v10) at ({360/10 *3 }:7cm) {};
		\node[vertex] (v9) at ({360/10 *4 }:7cm) {};
		\node[vertex] (v8) at ({360/10 *5}:7cm) {};
		\node[vertex] (v7) at ({360/10 *6 }:7cm) {};
		\node[vertex] (v6) at ({360/10 *7}:7cm) {};
		\node[vertex] (v5) at ({360/10 *8}:7cm) {};
		\node[vertex] (v4) at ({360/10 *9}:7cm) {};
        \draw ({360/10 *0 +1.5}:7cm) arc  ({360/10*0+1.5}:{360/10*1 -1.5 }:7cm);
        \draw[dotted] ({360/10 *1 +1.5}:7cm) arc  ({360/10*1+1.5}:{360/10*2 -1.5 }:7cm);
        \draw ({360/10 *2 +1.5}:7cm) arc  ({360/10*2+1.5}:{360/10*3 -1.5 }:7cm);
        \draw[dotted] ({360/10 *3 +1.5}:7cm) arc  ({360/10*3+1.5}:{360/10*4 -1.5 }:7cm);
        \draw ({360/10 *4 +1.5}:7cm) arc  ({360/10*4+1.5}:{360/10*5 -1.5 }:7cm);
        \draw[dotted] ({360/10 *5 +1.5}:7cm) arc  ({360/10*5+1.5}:{360/10*6 -1.5 }:7cm);
        \draw ({360/10 *6 +1.5}:7cm) arc  ({360/10*6+1.5}:{360/10*7 -1.5 }:7cm);
        \draw[dotted] ({360/10 *7 +1.5}:7cm) arc  ({360/10*7+1.5}:{360/10*8 -1.5 }:7cm);
        \draw ({360/10 *8 +1.5}:7cm) arc  ({360/10*8+1.5}:{360/10*9 -1.5 }:7cm);
        \draw[dotted] ({360/10 *9 +1.5}:7cm) arc  ({360/10*9+1.5}:{360/10*10 -1.5 }:7cm);
        
%        \node[xshift=0.2cm, yshift=0.15cm] at ({360/10*0 }:1.5cm) {$v$};
%        \node[xshift=0.25cm] at ({360/10*0 }:7cm) {$v_3$};
%		\node[xshift=0.25cm, yshift=0.1cm] at ({360/10*1 }:7cm) {$v_2$};
%		\node[xshift=0.1cm, yshift=0.25cm] at ({360/10 *2 }:7cm) {$v_1$};
%		\node[xshift=-0.1cm, yshift=0.25cm] at ({360/10 *3 }:7cm) {$v_{10}$};
%		\node[xshift=-0.2cm, yshift=0.15cm] at ({360/10 *4 }:7cm) {$v_9$};
%		\node[xshift=-0.25cm] at ({360/10 *5}:7cm) {$v_8$};
%		\node[xshift=-0.25cm, yshift=-0.15cm] at ({360/10 *6 }:7cm) {$v_7$};
%		\node[xshift=-0.1cm, yshift=-0.3cm] at ({360/10 *7}:7cm) {$v_6$};
%		\node[xshift=0.05cm, yshift=-0.3cm] at ({360/10 *8}:7cm) {$v_5$};
%		\node[xshift=0.2cm, yshift=-0.2cm] at ({360/10 *9}:7cm) {$v_4$};
     \foreach \from/\to in {v0/v1,v0/v2,v0/v3,v0/v4,v0/v5,v0/v6,v0/v7,v0/v8,v0/v9,v0/v10} \draw (\from) -- (\to);
   \end{scope}     
        
	\end{tikzpicture}
\end{minipage}
\caption{Switching non-isomorphic signed $W_{10}$ with exactly five negative edges}\label{psi_5(10)}
\end{figure}
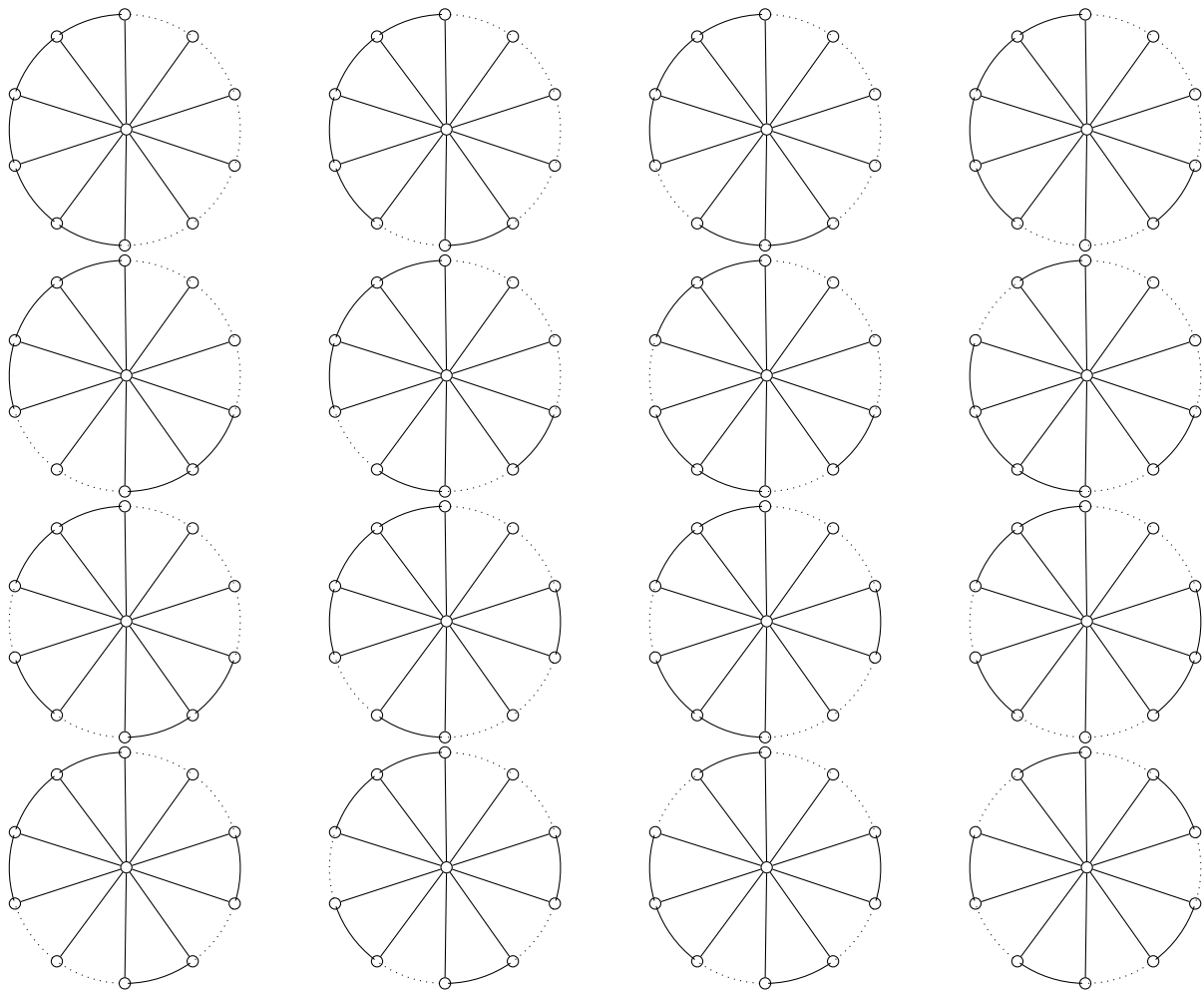

\end{document}